\documentclass[12pt]{article}
\usepackage{latexsym}
\usepackage{amsmath,amsthm}
\usepackage{tikz}

\usetikzlibrary{arrows}
\usetikzlibrary{decorations.markings}

\usepackage{amssymb,amsfonts}
\usepackage{color}
\usepackage{subfigure}
\usepackage{url}
\usepackage{graphicx}
\usepackage{epsfig}
\usepackage{mathrsfs}

    \newenvironment{indent_paragraph}[1]%
     {\vspace{-0.3cm}
     \begin{list}{}%
             {\setlength{\leftmargin}{#1}}%
             \item[]%
     }
     {\end{list}\vspace{-0.3cm}}

\newcommand{\R}{\mathbb{R}}
\newcommand{\real}{\mathbb{R}}

\newcommand{\commentout}[1]{}

\newcommand{\PP}{\mathcal P}

\newtheorem{lemma}{Lemma}

\newtheorem{remark}{Remark}

\begin{document}

\title{Confinement for Repulsive-Attractive Kernels}

\author{D. Balagu\'e $^1$, J. A. Carrillo$^2$, and Y. Yao$^3$}

\date{}

\maketitle

\begin{center}
\small{$^1$ Departament de Matem\`a\-ti\-ques,
Universitat Aut\`onoma de Barcelona,\\ E-08193 Bellaterra, Spain.}\\
\small{Email: \texttt{dbalague@mat.uab.cat}}

\

\small{$^2$ Department of Mathematics, Imperial College
London, London SW7 2AZ, UK.\\
\small{Email: \texttt{carrillo@imperial.ac.uk}}}

\

\small{$^3$ Department of Mathematics, University of Wisconsin, Madison, WI 53706,  USA.}\\
\small{Email: \texttt{yaoyao@math.wisc.edu}}
\end{center}

\begin{abstract}
We investigate the confinement properties of solutions of the
aggregation equation with repulsive-attractive potentials. We show
that solutions remain compactly supported in a large fixed ball
depending on the initial data and the potential. The arguments
apply to the functional setting of probability measures with
mildly singular repulsive-attractive potentials and to the
functional setting of smooth solutions with a potential being the
sum of the Newtonian repulsion at the origin and a smooth suitably
growing at infinity attractive potential.
\end{abstract}

\section{Introduction}
In this paper, we want to address ``confinement'' properties of
solutions to the nonlocal interaction equation
\begin{equation}
\rho_t = \nabla \cdot (\rho (\nabla W*\rho)), \qquad x \in \R^N,
t>0, \label{pde}
\end{equation}
with compactly supported initial data $\rho(0,\cdot)=\rho_0$ in a
functional space to be specified. These nonlocal equations appear
in many instances of mathematical biology \cite{ME99,TB,
TBL,F-actin, Stevens1,Stevens2}, mathematical physics
\cite{BenedettoCagliotiPulvirenti97,CMRV2,tosc:gran:00,LT}, and
material science \cite{Weinan:1994,SandierSerfaty,
SandierSerfatybook,LinZhang, AS,AMS,Mainini,DZ, Masmoudi}. They
are minimal models for the interaction of particles/agents through
pairwise potentials.

We say that the nonlocal equation \eqref{pde} satisfies a
confinement property in certain functional setting if every
solution $\rho(t,\cdot)$ to \eqref{pde} in that setting with
compactly supported initial data $\rho_0$ is compactly supported
for all times and its support lies in a fixed ball whose radius
only depends on $\rho_0$ and $W$.

In most of the mentioned applications, particles/agents repel to
each other in a short length scale while there is an overall
attraction in larger length scales. Therefore, we can typically
concentrate on repulsive-attractive potentials $W$ as in
\cite{KSUB,BUKB,FHK,soccerball,KHP,HuiUminskyBertozzi,Raoul,
FellnerRaoul1,FellnerRaoul2,BCLR,BCLR2}. These potentials lead to
a rich ensemble of compactly supported steady states whose
stability has recently been analyzed
\cite{KSUB,BUKB,soccerball,BCLR,BCLR2}. While the existence of
these compactly supported stable stationary states is a good
indication of confinement properties for these
repulsive-attractive potentials, it is not equivalent to
confinement. Let us finally mention, that repulsive-attractive
potentials have also been used in second order models for swarming
\cite{Dorsogna,Dorsogna2,Dorsogna3,CMP} where exponential decaying
at infinity potentials are more suitable from the modelling
viewpoint.

Confinement properties were addressed in \cite{CDFLS2} taking
advantage of the well-posedness theory for weak measure solutions
of \eqref{pde} developed in \cite{CDFLS}. Using the continuity
with respect to initial data in the functional setting of
probability measures $\PP (\R^N)$, the authors reduced the
confinement of the solutions to \eqref{pde} to a similar
confinement property for solutions of the associated particle ode
system:
\begin{equation}\label{odes}
\dot x_i = -\sum_{j\in Z(i)} m_j \nabla W(x_i - x_j), \quad i =
1,\ldots, n,
\end{equation}
where $Z(i) = \{j \in \{1, \ldots, n\}: j\neq i, x_j(t) \neq
x_i(t)\}$, $x_i(0)\in \R^N$ for all $i=1,\ldots,n$, $0\leq m_i\leq
1$, and $\sum_i m_i =M\geq 0$. The authors obtain a confinement
property in probability measures assuming that the potential is
radial and attractive outside a ball, apart from other technical
assumptions related to the well-posedness theory for probability
measures. We will improve over the main result in \cite{CDFLS2} in
terms of the assumed attractive strength at infinity. More precisely, we
will allow for slower growing at infinity potentials, see Section
2 for the precise hypotheses.

We will later obtain a confinement property for solutions to
\eqref{pde} in a smooth functional setting of compactly supported
initial data in $\mathcal{W}^{2,\infty}(\R^N)$. The trade-off is
to allow more singular repulsive at the origin potentials. In
fact, we will concentrate on the particular case of Newtonian
repulsion plus smooth attractive potential with certain growth at
infinity. In this functional setting, we can deal with smooth
solutions obtained by a slight variation of the arguments in
\cite{BLR,BalagueCarrillo,BLL,BCLR}. Section 3 shows that the same
ideas used for particles and solutions in the functional setting
$\PP (\R^N)$ apply in the continuum model \eqref{pde} for smooth
solutions with this particular potential. The strong Newtonian
repulsion at the origin of the potential allows us to derive a
priori $L^\infty$ bounds that otherwise are not known. Let us
finally mention that confinement for the repulsive Newtonian plus
an attractive harmonic potential was obtained in \cite{BLL}.

In the last section of the paper, Section 4, we perform some numerical
computations for \eqref{odes} with different repulsive-attractive
potentials and we study the confinement of the stationary states for
these cases. The stationary states are found using numerical techniques
 similar to the ones used in \cite{BUKB,soccerball}. We do a careful
study in each case by computing the radius of the support in order
to verify numerically the confinement for the solutions. Our
numerical studies indicate that the assumptions we impose on the
potentials are not sharp and could possibly be improved.

\section{Confinement for probability measures}

In this section, we will work with the theory developed in
\cite{CDFLS} in the framework of optimal mass transportation
theory applied to \eqref{pde}. We remind the reader that the
equation \eqref{pde} can be classically understood as the gradient
flow of the interaction potential energy \cite{CMRV2,AGS,CMRV1}.
Optimal transport techniques allow to construct a well-posedness
theory in the space of probability measures with bounded second
moments $\PP_2 (\R^N)$ at least for smooth potentials \cite{AGS}.
The regularity assumptions on the potential were relaxed in
\cite{CDFLS} allowing for potentials attractive at the origin with
a at most Lipschitz singularity there, i.e., allowing for local
behaviors like $W(x)\simeq |x|^a$, with $1\leq a<2$.

More precisely, we assume that the potential $W(x)$, see
\cite{CDFLS}, satisfies
\begin{indent_paragraph}{1cm}
\textbf{(NL0)} $W\in C(\mathbb{R}^N)\cap C^1(\mathbb{R}^N \backslash \{0\})$, $W(x) = W(-x)$, and $W(0)=0$.\\
\textbf{(NL1)} $W$ is $\lambda$-convex for a certain $\lambda\in\mathbb{R}$, i.e. $W(x)-\frac{\lambda}{2}|x|^2$ is convex.\\
\textbf{(NL2)} There exists a constant $C>0$ such that
$$W(z) \leq C(1+|z|^2), \quad \text{ for all }z\in \mathbb{R^d}.$$
\end{indent_paragraph}
to derive the well-posedness theory of gradient flow solutions to
\eqref{pde} with initial data in $\PP_2 (\R^N)$.

Under this set of assumptions \textbf{(NL0)}--\textbf{(NL2)}, we
can derive from \cite[Theorems 2.12 and 2.13]{CDFLS} that the
mean-field limit associated to the model in \eqref{pde} holds. On
one hand, this means that approximating the initial data by atomic
measures, we can approximate generic solutions of \eqref{pde} by
particular solutions corresponding to initial data composed by
finite number of atoms (particle solutions). On the other hand,
this also implies that the solution of \eqref{pde} given in
\cite[Theorems 2.12 and 2.13]{CDFLS} coincides with the atomic
measure constructed by evolving the locations of the atoms through
the ODE system \eqref{odes}. In other words, if one is interested
in showing a confinement property for \eqref{pde}, it suffices to
prove the confinement property for the particle system solving
\eqref{odes} since the solutions of the particle system
\eqref{odes} approximate accurately in finite time intervals the
solutions of the partial differential equation \eqref{pde}. All
these details are fully explained in \cite[Section 3]{CDFLS2}
allowing us to reduce directly to particle solutions.

To show confinement, we need additional assumptions on $W(x)$ as
in \cite{CDFLS2}. Throughout this paper, we assume that $W$ is
radially symmetric, and attractive outside some ball, i.e.,
\begin{indent_paragraph}{1cm}
\textbf{(NL-RAD)} $W$ is radial, i.e. $W(x)=w(|x|)$, and there exists
$R_a \geq 0$ such that $w'(r)\geq 0$ for
$r\geq R_a$.
\end{indent_paragraph}
It is pointed out in \cite[Remark 1.1]{CDFLS} that \textbf{(NL1)}
guarantees that the repulsive force cannot be too strong, more
precisely,
\begin{equation}\label{def_cw}
C_W := \sup_{x\in B(0, R_a)\backslash\{0\}} |\nabla W(x)|
\end{equation}
is bounded above. Here, we take the convention $C_W=0$ in case
$R_a=0$.

In order to prove confinement results, we need some other
condition to ensure that the attractive strength does not decay
too fast at infinity. In addition to \textbf{(NL0)-(NL3)} and
\textbf{(NL-RAD)}, we assume that $W$ satisfies the following
confinement condition:
\begin{indent_paragraph}{1cm}
\textbf{(NL-CONF)}\vspace{-0.6cm}
$$
 \lim_{r\to\infty} w'(r) r = +\infty,
$$
\end{indent_paragraph}
which is less restrictive than the assumption in \cite{CDFLS2},
namely $\displaystyle\lim_{r\to\infty} w'(r) \sqrt{r} = +\infty$.

Therefore, our goal in this section is to show that if the
particles interact under a potential satisfying assumptions
\textbf{(NL0)-(NL3)}, \textbf{(NL-RAD)}, and \textbf{(NL-CONF)},
have total mass $M=1$, center of mass at $0$, and are initially
confined in $B(0,\bar R_0)$, then they will be confined in some
ball $B(0,\bar R)$ for all times, where $\bar R$ is independent of
the number of particles $n$ but only depending on the kernel $W$ and the initial
support of the cloud of particles $\bar R_0$. Note that the zero
center of mass assumption is possible due to the translational
invariance of \eqref{pde} and \eqref{odes}. Note also that the
solution to the particle system \eqref{odes} in the sense of
\cite[Remark 2.1]{CDFLS2} might lead to a finite number of
collision times, in which the solution may lose its regularity.
Hence when we study the evolution of some quantities in time, we
only take the time derivative in the time intervals in which the
solution is regular.

The strategy to get confinement for particles is as follows: we
need to control quantities that quantifies how much the
distribution spreads in time. In \cite[Proposition 4.2]{CDFLS2}
the argument was based in following the particle furthest away
from the origin and use some energetic arguments to control the
mass of the particles nearby pushing the furthest particle. Here
we follow a different idea. We consider other moments of the
particle system to control the spread of the distribution of
particles in conjunction with the evolution of the furthest
particle from the center of mass. More precisely, we couple the
evolution of the third absolute moment of the particle system with the
evolution of the furthest particle.

\subsection{Evolution of the third moment}

Let $M_3(t)$ denote the third absolute moment of the particle
system \eqref{odes}, namely
\begin{equation*}
M_3(t) := \sum_{i=1}^n m_i |x_i|^3.
\end{equation*}
In this section our main goal is to estimate the time derivative
of $M_3(t)$. As we discussed before, there might be a finite
number of collision times in which $M_3(t)$ becomes
non-differentiable. Nevertheless, since all the particles have
finite velocity, $M_3(t)$ is Lipschitz continuous in time even
during collision. In all the computation below, the time
derivative of $M_3$ is only taken in the time intervals where
$M_3(t)$ is differentiable; and the continuity of $M_3(t)$ ensures
that the fundamental theorem of calculus still holds for $M_3(t)$.

Since $|x|^3$ is a convex function on $\mathbb{R}^N$,
we know that every pair of attracting particles would give a
negative contribution to $dM_3/dt$, whereas every  repulsing pair
gives a positive contribution. We can directly evaluate $dM_3/dt$
as follows:
\begin{equation}\label{dm}
\begin{split}
\frac{dM_3(t)}{dt} =~&3 \sum_i m_i |x_i| \langle\dot x_i , x_i\rangle\\
=~&3 \sum_i \sum_{j\in Z(i)} m_i   m_j \langle -\nabla W(x_i - x_j) , x_i |x_i|\rangle \\
=~& \frac{3}{2}\sum_i\sum_{j\in Z(i)} m_i   m_j \langle -\nabla W(x_i - x_j) , \big(x_i |x_i|-x_j |x_j|\big)\rangle \\
=~&  -\frac{3}{2}\sum_i\sum_{j\in Z(i)} m_i   m_j w'(|x_i - x_j|)
T_{ij},
\end{split}
\end{equation}
with $T_{ij}:= \dfrac{x_i-x_j}{|x_i-x_j|} \cdot \big(x_i |x_i|-
x_j |x_j|\big) $ and where antisymmetry of $\nabla W(x)$ is
used. Elementary manipulations yield that $T_{ij}$ can be
rewritten as
\begin{equation*}
\begin{split}
T_{ij} =(|x_i|+|x_j|) \frac{\frac{1}{2}(|x_i| - |x_j|)^2 +
\frac{1}{2} |x_i - x_j|^2}{|x_i - x_j|},
\end{split}
\end{equation*}
which gives the following upper and lower bound for $T_{ij}$:
\begin{equation}\label{ineq_T}
\frac{1}{2} \big(|x_i| + |x_j| \big) \big|x_i-x_j \big |\leq
T_{ij} \leq \big(|x_i| + |x_j| \big) \big |x_i-x_j \big | .
\end{equation}
Next we will find an upper bound for $-w'(|x_i-x_j|)$ in
\eqref{dm}.  Let us define the nearest particles set $N(i)$ as
the set of indexes of particles that are possibly repelling
the $i$-th particle, more precisely,
$$
N(i) := \Big\{j \in \{1, \ldots, n\}: 0< |x_j(t) - x_i(t)| \leq R_a\Big\}.
$$
Then \eqref{def_cw} implies that $-w'(|x_i-x_j|)\leq C_W$ for all
$j\in N(i)$.

For $j \not\in N(i)$, \textbf{(NL-RAD)} gives that $-w'(|x_i-x_j|)
\leq 0$. A better bound can be obtained using \textbf{(NL-CONF)}:
note that for any fixed constant $K_1>0$ to be specified momentarily,
there exists some $R_{K_1}>2R_a$, such that
\begin{equation*}
-w'(r) < -\frac{K_1}{r}\quad  \text{ for all } r>R_{K_1}.
\end{equation*}
Let us define the set of furthest particles $F(i)$ as the
set of indexes of particles whose distance to the $i$-th particle
are larger than $R_{K_1}$, namely
$$
F(i) := \Big\{j \in \{1, \ldots, n\}: |x_j(t) - x_i(t)| > R_{K_1}\Big\}.
$$
The definitions of $N(i)$ and $F(i)$ are illustrated in Figure \ref{def_ni_fi}.
 Then the upper bound for $-w'(|x_i-x_j|)$ can be summarized as following:
\begin{equation}\label{ineq_w'}
-w'(|x_i-x_j|) \leq \begin{cases}
C_W & \text{for } j \in N(i),\\
-\dfrac{K_1}{|x_i-x_j|}& \text{for }  j \in F(i),\\
0 & \text{for } j \not \in N(i) \cup F(i).
\end{cases}
\end{equation}

\tikzset{
  big arrow/.style={
    decoration={markings,
    mark=at position 0.12cm with {\arrow[scale=2,#1]{<}},
    mark=at position 1 with {\arrow[scale=2,#1]{>}}
    },
    postaction={decorate},
    shorten >=0.4pt},
  big arrow/.default=black}

  \tikzset{
  single big arrow/.style={
    decoration={markings,
    mark=at position 1 with {\arrow[scale=2,#1]{>}}
    },
    postaction={decorate},
    shorten >=0.4pt},
  big arrow/.default=black}

\begin{figure}[htb]
\begin{center}
\begin{tikzpicture}
\fill[blue!15] (9,4) rectangle (-2,-4);
\filldraw[fill=black, draw=black] (4,0) circle (0.06cm);
\filldraw [fill=white, draw=black](4,0) circle (3.7cm);
\filldraw [fill=red!15, draw=black](4,0) circle (1.5cm);
\draw [black, big arrow] (4,0) -- (4-3.7/1.41421,-3.7/1.41421) node[midway,above, xshift=-0.3cm, yshift=-0.2cm] {$R_{K_1}$};
\draw [big arrow] (4,0) -- (4+1.5/1.41421,-1.5/1.41421) node[midway,below,xshift=-0.1cm,yshift=0.15cm] {$R_a$};
\filldraw[fill=black, draw=black] (4,0) circle (0.06cm) node[xshift=0.35cm] {\large{$x_i$}};
\draw (0,2) node[blue!50!black] {\Large{$F(i)$}};
\draw (4,0.7) node[red!50!black] {\Large{$N(i)$}};
 \end{tikzpicture}
\end{center}
\caption{Illustration of the sets $N(i)$ and $F(i)$. For the $i$-th
particle, $N(i)$ is defined as the set of indexes of particles in
the red region, while $F(i)$ is the set of indexes of particles in
the blue region.} \label{def_ni_fi}
\end{figure}
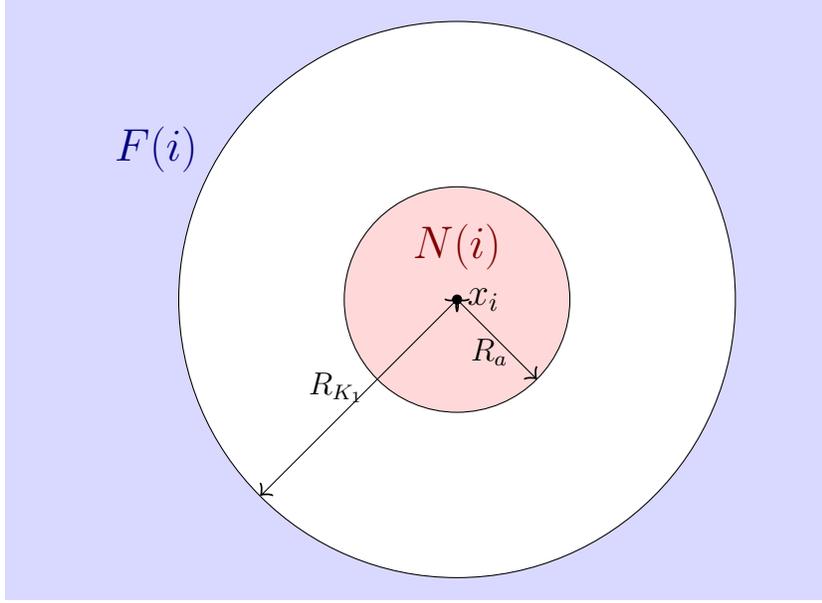

By plugging \eqref{ineq_T} and \eqref{ineq_w'} into \eqref{dm} and
setting $K_1:= 10C_W R_a$, we obtain
\begin{align}
\frac{dM_3(t)}{dt}
\leq ~&  \frac{3}{2}\sum_i m_i \Big( \sum_{j\in N(i)}  m_j C_W R_a  \big(|x_i| + |x_j| \big) - \sum_{j\in F(i)} m_j K_1 \frac{1}{2}(|x_i|+|x_j|) \Big) \nonumber\\
\leq ~& \frac{3}{2} C_W R_a \sum_i m_i
\big(T_r^{i}-5T_a^{i}\big),\label{dm_2}
\end{align}
with
$$
T_r^{i}:=\sum_{j\in N(i)}  m_j   \big(|x_i| + |x_j| \big) \qquad
\mbox{ and } \qquad T_a^{i}:=\sum_{j\in F(i)} m_j
(|x_i|+|x_j|)\,.
$$

Now we claim that
\begin{equation}\label{claim_t}
T_r^{i} \leq 4 T_a^{i} \quad \text{ if } |x_i| > R_{K_1}.
\end{equation}
Its validity is one of the main reasons for imposing the requirement
\textbf{(NL-CONF)}. To prove the claim, recall that
we assume the center of mass is at 0 at $t=0$ without loss of generality. Due
to the conservation of the center of mass, for any time $t$, we have $x_j(t)$ satisfies
$\sum_j m_j x_j(t) = 0$. Let $e_i \in \mathbb{R}^N$ denote the unit vector pointing
in the direction of $x_i$, then it follows immediately that
$$
\sum_{j=1}^n m_j x_j \cdot e_i = 0.
$$
For $|x_i| > R_{K_1}$, we split the above sum into three parts, and get
\begin{align}
\sum_{j \in N(i)} m_j x_j \cdot e_i &= -\sum_{j \in F(i)} m_j x_j \cdot e_i - \sum_{j\not \in N(i) \cup F(i)} m_j x_j \cdot e_i \nonumber \\
&\leq -\sum_{j \in F(i)} m_j x_j \cdot e_i \leq T_a^{i},
\label{center_mass_ineq}
\end{align}
where the first inequality is due to the fact that $x_j\cdot e_i > 0$ for all $j \not\in F(i)$.

Moreover, recall that we find $R_{K_1}$, we force it to be bigger
than $2R_a$. This is to guarantee that for all $|x_i| > R_{K_1} >
2R_a$ and $j\in N(i)$, the angle between the vectors $x_i$ and
$x_j$ is less than $\frac{\pi}{6}$. As a result, we have $x_j
\cdot e_i \geq \frac{\sqrt{3}}{2} |x_j|$. Noticing that for
$|x_i|> R_{K_1}$ and $j\in N(i)$ we also have $|x_j|>|x_i|/2$,
which is equivalent with $|x_j| > \frac{1}{3}(|x_i|+|x_j|)$. Thus
finally we have
$$
\sum_{j \in N(i)} m_j x_j \cdot e_i \geq \sum_{j \in N(i)}  \frac{m_j}{2\sqrt{3}} (|x_i|+|x_j|) = \frac{T_r^i}{2\sqrt{3} } \quad\text{ for } |x_i|>R_{K_1},
$$
and by combining it with \eqref{center_mass_ineq} we obtain the
claim \eqref{claim_t}.

Due to \eqref{claim_t}, we deduce that for any $|x_i| > R_{K_1}$, $T_r^{i}
- 5 T_a^{i} \leq -T_a^{i}$, hence \eqref{dm_2} becomes
\begin{equation*}
\begin{split}
\frac{dM_3(t)}{dt} \leq \frac{3}{2} C_W R_a \Big( T_1 -\sum_{|x_i|
> R_{K_1}} \sum_{j\in F(i)} m_i m_j (|x_i| + |x_j|)\Big)\,,
\end{split}
\end{equation*}
with
$$
T_1:=\sum_{|x_i| \leq R_{K_1}} \sum_{j\in N(i)} m_i m_j (|x_i| +
|x_j|)\,.
$$
Note that we can easily bound $T_1$ by a constant only depending
on $W$ (since $|x_i| \leq R_{K_1}$ and $|x_j| \leq R_{K_1} + R_a$), thus we can rewrite the above inequality as inequality as
\begin{equation}\label{dm_3}
\begin{split}
\frac{dM_3(t)}{dt} \leq C_1 - C_2\sum_{|x_i| > R_{K_1}} \Big( m_i |x_i|
\sum_{j\in F(i)} m_j \Big),
\end{split}
\end{equation}
where $C_1, C_2$ only depends on $W$. At this point we will take a
pause on the evolution of $M_3$; we will revisit the inequality
\eqref{dm_3} soon in Section \ref{couple_discrete} to couple it
with the evolution of the support.

\subsection{Coupling with the evolution of the support}
\label{couple_discrete}
For all $t\geq 0$, let  $R(t)$ denote the distance of the furthest
particle from the center of mass (which we assumed to be 0 without
loss of generality), namely
\begin{equation}\label{def:R}
R(t) := \max_{i=1,\ldots,n}|x_i(t)|.
\end{equation}
It is pointed out in \cite[Proposition 4.2]{CDFLS} that $R(t)$ is
Lipschitz in time. Our goal is  to prove that
$\displaystyle\limsup_{t\to \infty} R(t) < \bar R$, where $\bar R$
only depends on $W$.

We begin by reminding a claim proved in \cite[Proposition
2.2]{CDFLS2}: Let $e$ be any unit vector. Then
\begin{equation}\label{ineq:one_third_mass}
\sum_{x_j \cdot e \leq R(t)/2} m_j \geq \frac{1}{3}\sum_{j} m_j,
\end{equation}
i.e. the green region in Figure \ref{one_third_mass} contains at
least $1/3$ of the total mass.

\begin{figure}
\begin{center}
\begin{tikzpicture}
\def\radius{3.7cm}
\filldraw [fill=white, draw=black](0,0) circle (\radius);
\begin{scope}[fill=green!15]
\clip (0,0) circle(\radius);
\fill(-\radius,-\radius) rectangle (\radius/2,\radius);
\end{scope}
\draw [black, big arrow] (0,0) -- (\radius/2, 1.732*\radius/2) node[midway,above, xshift=-0.3cm, yshift=-0.2cm] {$R(t)$};
\draw [big arrow] (0,0) -- (\radius/2, 0) node[midway,above,xshift=-0.1cm,yshift=-0.15cm] {$R(t)/2$};
\draw [thick, dashed] (\radius/2,\radius) -- (\radius/2,-\radius) ;
\filldraw[fill=black, draw=black] (0,0) circle (0.06cm) node[xshift=-0.3cm,yshift=-0.2cm] {$0$};

\draw [thick, single big arrow] (0,-1) -- (\radius/3, -1) node[midway,below,xshift=-0.1cm,yshift=-0.15cm] {unit vector $e$};

 \end{tikzpicture}
\end{center}
\caption{For any unit vector $e$, the green region above contains
at least one third of the total mass. Here $R(t)$ is as defined in
\eqref{def:R}.} \label{one_third_mass}
\end{figure}
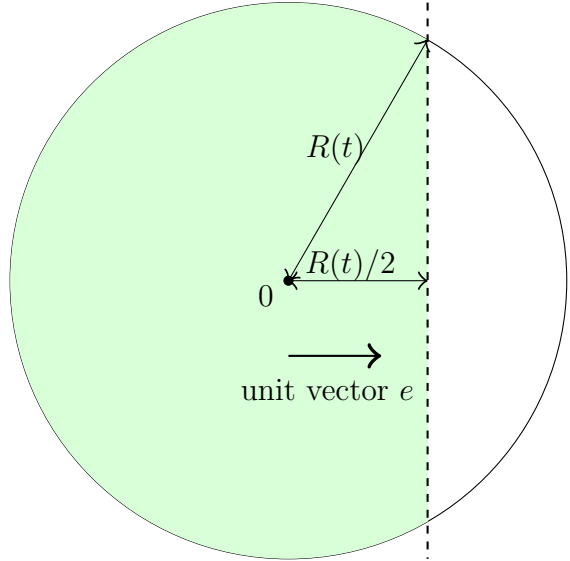

It is argued in the proof of \cite[Proposition 4.2]{CDFLS} that
for all time $t>0$, there is a particle index $i_0(t)$ (here $i_0$
may depend on $t$), such that
\begin{equation}\label{jjj}
|x_{i_0}(t)| = R(t) \qquad \text{ and } \qquad \frac{d^+}{dt} R(t)
= \dot x_{i_0}(t)\cdot \frac{x_{i_0}(t)}{R(t)},
\end{equation}
where $\frac{d^+}{dt}$ stands for the right derivative. This
technical point is due again to the lack of regularity of $R(t)$
for all $t$, see \cite[Proposition 4.2]{CDFLS}. From now on, the
index $i_0$ refers to any index satisfying the previous
properties.

To control $\frac{d^+}{dt} R(t)$, it suffices to look at the outward
velocity of the $i_0$-th particle at this time. As argued in
\cite[Proposition 2.2]{CDFLS2} and illustrated in Figure \ref{in_n_out},
the red region is possibly pushing it out, but all the green
region is pulling it towards the origin. We proceed by
estimating the compensation between these two competing effects. Let
$G(i_0)$ denote the set of indexes of particles in the green region in
Figure \ref{in_n_out}, namely
$$
G(i_0) := \Big\{j \in \{1, \ldots, n\}: x_j(t) \cdot
\frac{x_{i_0}}{|x_{i_0}|} \leq \frac{R(t)}{2}\Big\}\,.
$$
Recall that throughout this section, we assume the total mass is 1
without loss of generality. It then follows from
\eqref{ineq:one_third_mass} that
\begin{equation}\label{massgi}
\sum_{j\in G(i_0)} m_j \geq1/3\,.
\end{equation}

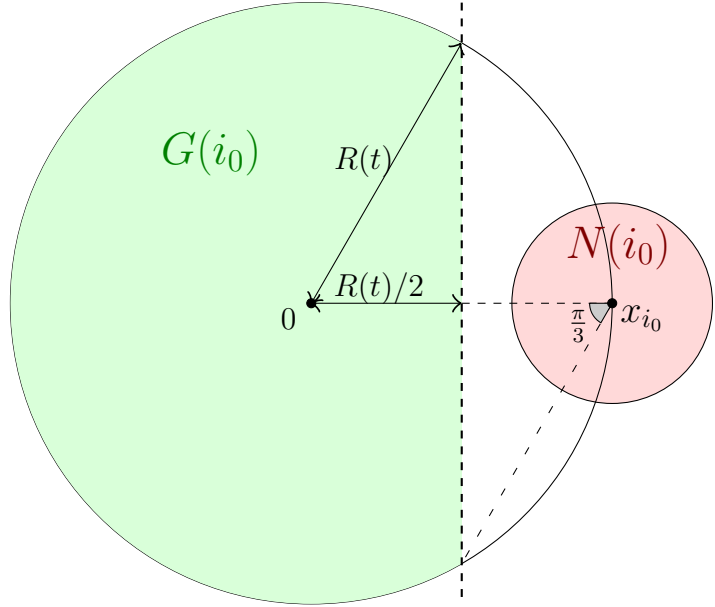
\begin{figure}[htb]
\begin{center}
\begin{tikzpicture}
\def\radius{4cm}
\filldraw [fill = red!15, draw=black](\radius,0) circle (\radius/3);
\draw [black](0,0) circle (\radius);
\begin{scope}[fill=green!15]
\clip (0,0) circle(\radius);
\fill(-\radius,-\radius) rectangle (\radius/2,\radius);
\end{scope}
\draw [black, big arrow] (0,0) -- (\radius/2, 1.732*\radius/2) node[midway,above, xshift=-0.3cm, yshift=-0.2cm] {$R(t)$};
\draw [big arrow] (0,0) -- (\radius/2, 0) node[midway,above,xshift=-0.1cm,yshift=-0.15cm] {$R(t)/2$};
\draw [thick, dashed] (\radius/2,\radius) -- (\radius/2,-\radius) ;
\draw [black, dashed] (\radius/2,0) -- (\radius,0);
\draw [black, loosely dashed] (\radius, 0)-- (\radius/2, -1.732*\radius/2) ;
\filldraw[fill=gray!40] (\radius-3mm,0) arc (180:240:3mm) -- (\radius, 0) -- (\radius-3mm, 0);
\draw (\radius, 0) node[xshift=-0.45cm, yshift=-0.3cm] {$\frac{\pi}{3}$};
\draw[green!50!black] (-\radius/3, \radius/2) node {\Large{$G(i_0)$}};
\draw[red!50!black] (\radius*1.02, \radius*0.19) node {\Large{$N(i_0)$}};

\filldraw[fill=black, draw=black] (0,0) circle (0.06cm) node[xshift=-0.3cm,yshift=-0.2cm] {$0$};
\filldraw[fill=black, draw=black] (\radius,0) circle (0.06cm) node[xshift=0.4cm,yshift=-0.2cm] {\large{$x_{i_0}$}};

 \end{tikzpicture}
\end{center}
\caption{When $R(t) > 2R_a$, the particles in the green region are all pulling the
particle $x_{i_0}$ towards the origin, while the particles in the
red region are possibly pushing it out.} \label{in_n_out}
\end{figure}
Using \eqref{jjj}, \eqref{odes}, and some simple manipulations,
the growth of $R(t)$ is controlled by the following inequality:
\begin{align}
\frac{d^+}{dt} R(t) \leq~ & C_W \sum_{j\in N(i_0)} m_j - \sum_{j \in G(i_0)} m_j w'(|x_j - x_{i_0}|)  \cos(\theta(-x_{i_0}, x_j - x_{i_0})),\nonumber\\
\leq~& C_W \sum_{j\in N(i_0)} m_j - \frac{1}{2}\sum_{j \in G(i_0)}
m_j w'(|x_j - x_{i_0}|), \label{dr}
\end{align}
where $\theta(-x_{i_0}, x_j - x_{i_0})$ denotes the angle between the two vectors
$-x_{i_0}$ and $x_j - x_{i_0}$, which is less than $\frac{\pi}{3}$
as shown in Figure \ref{in_n_out}. Due to \textbf{(NL-CONF)}, for any
large constant $K_2$, which we will fix later, there exists some
radius $R_{K_2}$, such that $w'(r) >K_2/r$ for all $r>R_{K_2}$.
Hence for all $t$ satisfying $R(t) > R_{K_2}$, \eqref{dr} becomes
\begin{equation}\label{dr_2}
\begin{split}
\frac{d^+}{dt} R(t) \leq~& C_W \sum_{j\in N(i_0)} m_j - \frac{1}{2} \frac{{K_2}}{\displaystyle\sup_{j\in G(i_0)} |x_j - x_{i_0}|} \sum_{j \in G(i_0)} m_j \\
\leq~& C_W \sum_{j\in N(i_0)} m_j - \frac{K_2}{12 R(t)}\,,
\end{split}
\end{equation}
where \eqref{massgi} was used to obtain the last inequality.

To ensure the coupling between the growth of $M_3(t)$ and the
growth of $R(t)$ go smoothly, let us go back to \eqref{dm_3} and
perform some elementary manipulation on it. When $R(t) >
R_{K_1}+R_a$, for any $j \in N(i_0)$, we have $|x_j|>R_{K_1}$,
hence
\begin{equation}\label{dm_4}
\begin{split}
\frac{dM_3(t)}{dt} \leq~& C_1 - C_2\sum_{i \in N(i_0)} \Big( m_i |x_i|
\sum_{j\in F(i)} m_j \Big)\,.
\end{split}
\end{equation}
And if in addition we have $R(t) > 2(R_{K_1}+R_a)$, then it follows that $G(i_0) \subset
F(i)$ for any $i \in N(i_0)$, hence we can replace the $F(i)$ in
\eqref{dm_4} by $G(i_0)$ and obtain
\begin{align}
\frac{dM_3(t)}{dt} \leq~& C_1 - C_2\sum_{i \in N(i_0)} \Big(m_i |x_i| \sum_{j\in G(i_0)} m_j\Big) \leq  C_1 - \frac{1}{3}C_2 (R(t)-R_a) \sum_{i \in N(i_0)} m_i \nonumber \\
=~& \left(C_1 + \frac{1}{3} C_2 R_a \right)  - \frac{1}{3}C_2
R(t) \sum_{i \in N(i_0)} m_i,  \label{dm_final}
\end{align}
where we used \eqref{massgi} again to obtain the second inequality.  Finally, we set $R_1 :=
\max\{R_{K_2}, 2(R_{K_1} + R_a)\}$, to ensure that both
\eqref{dr_2} and \eqref{dm_final} hold for $R(t)>R_1$.

Finally we are ready to couple $M_3(t)$ with $R(t)$. By putting the
estimates on $\frac{d^+}{dt}R(t)$ and $\frac{d}{dt}M(t)$ together,
we will show that if $R(t)$ grows from $R_1$ to some very large
number in some time interval $[t_1, t_2]$, then the integral of
$dM_3/dt$ is negative over this time interval, i.e.,
$M_3(t_2)<M_3(t_1)$. On the other hand, we will directly prove
that $M_3(t_2)$ must be bigger than $M_3(t_1)$, which causes a
contradiction.

Let $A_1$ be a sufficiently large constant which we will determine
later. If the particles start in $B(0,\bar R_0)$ and eventually
touch the boundary of $B(0, A_1R_1)$, then there exist
$0<t_1<t_2$, such that
$$
R(t_1) = R_1, ~~ R(t_2) = A_1 R_1, ~~\frac{d^+}{dt} R(t_2)\geq 0,
$$
and
$$
R(t)\in [R_1, A_1R_1]~~ \text{ for all } t_1\leq t \leq t_2.
$$
More precisely,  by letting $t_2 := \min \{ t\geq 0: R(t)=A_1
R_1\}>0$, and $t_1 := \max\{ 0\leq t\leq t_2: R(t) =
R_1\}$, they would satisfy all the requirements.

Since $R(t_2)>R(t_1)$, we have
\begin{equation*}
\int_{t_1}^{t_2} \Big(\frac{d^+}{dt}R(t)\Big) R(t)\, dt =
\frac{R^2(t)}{2}\Big|_{t_1}^{t_2}>0.
\end{equation*}
Using \eqref{dr_2}, the above inequality implies
\begin{equation*}
\int_{t_1}^{t_2} C_W \left( \sum_{j \in N(i_0(t))} m_j\right) R(t)
\,dt
> \int_{t_1}^{t_2} \frac{1}{12}K_2 \, dt,
\end{equation*}
and by plugging it into the integral version of \eqref{dm_final} we obtain
\begin{equation*}
\int_{t_1}^{t_2} \frac{dM_3}{dt}(t) \, dt \leq   \int_{t_1}^{t_2}
\left[(C_1 + \frac{1}{3} C_2 R_a)  - \frac{C_2}{36C_W}
K_2\right]\,dt,
\end{equation*}
hence by choosing $K_2:=1+ 36C_W (C_1 + \frac{1}{3}C_2 R_a)/C_2$,
which only depends on $W$, we have $M_3(t_2) < M_3(t_1)$.

On the other hand, if $R(t)$ successfully grows from $R_1$ to
$A_1 R_1$, we will show that $M_3$ indeed has to increase, namely
$M_3(t_2)>2M_3(t_1)$ for $A_1$ sufficiently large.  First, we can
bound $M_3(t_1)$ above by the very rough bound
$R(t_1)^3=R_1^3$. At time $t_2$, recall that
$\frac{d^+}{dt}R(t_2) \geq 0$, hence \eqref{dr_2} implies that
$$
\sum_{j\in N(i_0(t_2))} m_j \geq \frac{K_2}{12 C_W A_1 R_1}.
$$

Finally,  by noticing that
$$
M_3(t_2) \geq (R(t_2)-R_a)^3 \sum_{j\in N(i_0(t_2))} m_j\geq
\frac{K_2}{12 C_W A_1 R_1} (A_1 R_1-R_a)^3 \,,
$$
we obtain $M_3(t_2) \gtrsim (A_1R_1)^2$. Therefore we can choose
$A_1$ sufficiently large such that $M_3(t_2)
> 2M_3(t_1)$, which leads to a contradiction with
$M_3(t_2)<M_3(t_1)$.

Note that the proof above shows that $R(t)$ can never reaches
$\bar R := A_1 R_1$, which is a large constant only depend on
$W$, and in particular is independent of the number of particles.


\section{Confinement for kernel with Newtonian repulsion}
In this section, we consider the interaction kernel $W(x)$ given
by
\begin{equation}
W(x) = -\mathcal{N}(x) + W_a(x) \label{W_newton_rep}
\end{equation}
with $N\geq 2$. Here $\mathcal{N}(x)$ is the Newtonian kernel,
namely
$$
\mathcal{N}(x) =
\begin{cases}
\frac{1}{2\pi}\ln|x| & N=2,\\[0.1cm]
-\dfrac{c_N}{|x|^{N-2}} & N\geq 3,
\end{cases}
$$
where $c_N$ denotes the volume of a unit ball in $\mathbb{R}^N$. Throughout this section we assume that $W_a(x)$ satisfies the following assumptions:
\begin{indent_paragraph}{1cm}
\textbf{(W1)} $\Delta W_a \in L^1_{loc}(\mathbb{R}^N)$.\\[0.1cm]
\textbf{(W2)} $\Delta W_a$ is bounded in $\mathbb{R}^N \setminus B_\epsilon(0)$ for any $\epsilon>0$.\\[0.1cm]
\textbf{(W-RAD)} $W_a(x) = w(|x|)$ with $w\in C^1((0,\infty))$ and $w'(r)\geq 0$ for $r> 0$.\\[0.1cm]
\textbf{(W-CONF)}$\displaystyle\lim_{r\to\infty} w'(r) r^{1/N} = +\infty$. \\
\end{indent_paragraph}

Our goal in this section is to show that under the above
assumptions, if a solution has total mass $M=1$, center of mass at
$0$, and are initially confined in $B(0,\bar R_0)$, then it will
be confined in some fixed ball centered at 0 for all times, where
the radius of the ball only depends on $W_a$, $\bar R_0$, the
dimension $N$, and the $L^\infty$ norm of the initial data.

\begin{remark}
For $N=1$, the Newtonian kernel becomes $|x|$. Note that
in this case the confinement result does not hold under the
assumptions above, since the repulsive velocity field between two
particles will be a constant regardless of the distance between
them, while the attraction may vanish as the distance goes to
infinity. We can compensate this difficulty by imposing stronger
assumption on the attractiveness of $W_a$ at infinity. More
precisely, by replacing \emph{\textbf{(W-CONF)}} by
$\displaystyle\lim_{r\to\infty} (w'(r) - 1)r = +\infty$,  the
confinement result will hold with a similar proof as in Section
{\rm 2} carried over at the continuum level.\end{remark}

\begin{remark}
{\rm\textbf{(W-CONF)}} is more restrictive than
{\rm\textbf{(NL-CONF)}}, especially for large $N$. In the proof
below, one can see that $dM_3/dt$ does not cause a problem at all,
indeed it satisfies the same inequality as in the non-singular
kernel case in Section {\rm 2}. The problem lies in $dR/dt$: due
to the singular repulsive kernel, we got a worse control of
$dR/dt$ than before, see \eqref{ineq:drdt_newton_1}.
\end{remark}

We point out that slight variations of the arguments in
\cite[Section 5]{BCLR} and \cite{BalagueCarrillo,BLL} give a
well-posedness theory for smooth solutions constructed by
characteristics. More precisely, for any compactly supported
initial data $\rho_0\in \mathcal{W}^{2,\infty}(\real^N)$, there
exists a unique classical solution $\rho\in C^1([0,T]\times
\real^N)\cap
\mathcal{W}^{1,\infty}_{loc}(\real_+,\mathcal{W}^{1,\infty}(\real^N))$
to \eqref{pde} with $W$ satisfying \textbf{(W1)}-\textbf{(W2)}.
Moreover, the associated velocity field $v(t,x)=-\nabla W\ast
\rho$ is Lipschitz continuous in both space and time, hence the
characteristics are well defined:
\[
   \frac{d}{dt}X_t = -(\nabla W\ast \rho)(t,X_t),
\]
and the solution $\rho$ is given by
\[
   \rho(t,x) = \rho_0(X_{t}^{-1})\det (DX_t^{-1}).
\]
Since the initial data is compactly supported, it remains
compactly supported for all time (although the support may grow in time), and its support is obtained through
the $C^1$-characteristic maps $X_t$.

First we remind a lemma showing $L^\infty$-bounds of the solution.
This is a classical argument that can be seen for instance in
\cite{FellnerRaoul2,BLL} and \cite[Section 5]{BCLR} but we give a
short proof for completeness.

\begin{lemma}\label{lemma:u_infty}
Let $W$ be given by \eqref{W_newton_rep}, with $W_a$ satisfying
{\rm\textbf{(W1)}-\textbf{(W2)}}. Let $\rho$ be a classical solution
to \eqref{pde} with compactly supported initial data $\rho_0\in
\mathcal{W}^{2,\infty}(\real^N)$. Then $\|\rho(t,\cdot)\|_\infty
\leq M_0$ for all $t\geq 0$, where $M_0$ only depends on $W_a$ and
$\rho_0$.
\end{lemma}

\begin{proof}
Due to the assumption \textbf{(W1)}, we can find $r_0 > 0$
sufficiently small, such that
$$
\int_{B(0, r_0)} |\Delta W_a(x)| dx \leq \frac{1}{2}.
$$
Then it follows from \textbf{(W2)} that $M_W :=
\displaystyle\sup_{x\in \mathbb{R}^N \setminus B(0, r_0)} |\Delta
W_a(x)|$ is finite. We define $M_0$ as
$$
M_0:= \max\{2M_W  \|\rho_0\|_1  , \|\rho_0\|_\infty\}.
$$

Let us denote by $\tilde M(t)=\max_{x\in\mathbb{R}^N} \rho(t,x)$.
If the
desired result does not hold, then there exists some $t_1>0$, such
that $\tilde M(t_1)> M_0$ and $\tilde M(t)$ is increasing at $t=t_1$. This enables
us to find some $x_1 \in \mathbb{R}^N$, such that
$\tilde M(t_1)=\rho(t_1,x_1)$, and $\frac{\partial}{\partial
t}\rho(t_1,x_1) \geq 0$. On the other hand, since $\rho$ is a
classical solution, we have
\begin{equation}
\begin{split}
\frac{\partial}{\partial t}\rho(t_1,x_1) &= \nabla \rho \cdot (\nabla W*\rho) + \rho (\Delta W*\rho)\\
&= \rho(t_1,x_1) \Big((\rho*\Delta W_a)(t_1,x_1) - \rho(t_1,x_1)
\Big).
\end{split}
\label{eq:u_t}
\end{equation}
Now let us split the integral $\rho*\Delta W_a$ in $B(0, r_0)$ and
outside to get
\begin{equation*}
\begin{split}
(\rho*\Delta W_a)(t_1,x_1) &\leq
\frac{1}{2}\|\rho(t_1,\cdot)\|_\infty + M_W\|\rho(\cdot, 0)\|_1 =
\frac{1}{2} \rho(t_1,x_1) + M_W \|\rho_0\|_1 .
\end{split}
\end{equation*}
Plugging it into \eqref{eq:u_t}, we have
\begin{equation*}
\begin{split}
\frac{\partial}{\partial t}\rho(t_1,x_1)  &\leq \rho(t_1,x_1)
\Big( M_W \|\rho_0\|_1  - \frac{1}{2} \rho(t_1,x_1) \Big) \leq \rho(t_1,x_1)
\Big(\frac{M_0}{2} - \frac{\rho(t_1,x_1)}{2} \Big)< 0,
\end{split}
\end{equation*}
which contradicts with the assumption that
$\frac{\partial}{\partial t} \rho(t_1,x_1) \geq 0$.
\end{proof}

Next we present a technical lemma which will be used in the proof of confinement.

\begin{lemma}
\label{lemma:integral} Assume $u \in L^\infty(\mathbb{R}^N)$ with
$0\leq p < N$. Then it follows that
\begin{equation}
\label{ineq:rearrangement} \int_{B(x_0, R)} \frac{u(y)}{|x_0-y|^p}
dy \leq C(N,p) \|u\|_\infty^{p/N} \left(\int_{B(x_0, R)} u(y) dy
\right)^{(N-p)/N}
\end{equation}
for all $x_0\in\mathbb{R}^N$ and all $R > 0$.
\end{lemma}
\begin{proof}
First note that it suffices to prove the following
inequality holds for all $v \in L^1(\mathbb{R}^N) \cap
L^\infty(\mathbb{R}^N)$:
\begin{equation*}
\int_{\mathbb{R}^N} \frac{v(y)}{|y|^p} dy \leq C(N,p)
\|v\|_\infty^{p/N} \|v\|_1^{(N-p)/N},
\end{equation*}
by letting $v(y) = \chi_{B(0,R)}(y) u(y+x_0)$, where $\chi_\Omega$
is the indicator function on a set $\Omega \subset \mathbb{R}^N$.
We point out that one could use H\"older inequality and
interpolation inequality on weak $L^p$ spaces to obtain a slightly
weaker inequality than above, but we will use an easier and more
elementary approach instead.

Let $w$ be an indicator function taking value $\|v\|_\infty$ on
some disk centered at 0 and taking value $0$ outside, where the
size of the disk is chosen such that $w$ and $v$ have the same
$L^1$ norm. More precisely, $w$ is given by
\begin{equation*}
w := \|v\|_\infty\chi_{B(0,r_0)}, \text{ where } r_0
:= \Big( \frac{\|v\|_1}{c_N \|v\|_\infty}\Big)^{1/N},
\end{equation*}
here $c_N$ is the volume of the unit ball in $\mathbb{R}^N$. Since
$\|v\|_1 = \|w\|_1$, it is straightforward to verify that
\begin{equation}\label{ineq:v_vs_w}
\int_{B(0,r)} v(y) dy \leq \int_{B(0,r)} w(y) dy \text{ for all }r\geq 0.
\end{equation}
Now we start with the left hand side of
\eqref{ineq:rearrangement}, and Fubini's theorem yields that
\begin{align*}
\int_{\mathbb{R}^N}  \frac{v(y)}{|y|^p} dy &= \int_{\mathbb{R}^N} \int_0^\infty v(y) \chi_{\{t\leq |y|^{-p}\}} dt dy = \int_0^\infty \int_{B(0,t^{-1/p})} v(y) dy dt\\
&\leq \int_0^\infty \int_{B(0,t^{-1/p})} w(y) dy dt = \int_{\mathbb{R}^N}  \frac{w(y)}{|y|^p} dy \\
&=  \frac{ N c_N \|v\|_\infty}{N-p} |r_0|^{N-p} = \frac{
N c_N^{p/N} \|v\|_\infty^{p/N}}{N-p} \|v\|_1^{(N-p)/N}\,,
\end{align*}
where \eqref{ineq:v_vs_w} was used.
\end{proof}

\subsection{Evolution of the third absolute moment}
Similar to the particle system case, we also start with estimating
the time derivative of the third absolute moment $M_3$. Here the
third absolute moment $M_3$ is given by
$$
M_3(t) := \int_{\mathbb{R}^N} \rho(t,x) |x|^3 dx,
$$
and note that in the continuum setting $M_3$ is indeed
differentiable in time for all $t\geq 0$, since $\rho(t,x)$ is a
classical solution. The same computation as \eqref{dm} leads to
$$
\frac{dM_3(t)}{dt} = -\frac{3}{2} \int_{\mathbb{R}^N}
\int_{\mathbb{R}^N} \rho(t,x)\rho(t,y) w'(|x-y|)
\frac{x-y}{|x-y|}\cdot (x|x|-y|y|) dydx.
$$
Due to the assumptions \textbf{(W-RAD)} and \textbf{(W-CONF)}  on
$W_a$, for any $A>0$ (which will be fixed at the end of this
subsection), there exists some $R_A>1$, such that the following
bound for  $-w'(r)$ holds, where $C_N$ is some constant only depending on $N$:
\begin{equation*}
-w'(r) \leq \begin{cases}
\dfrac{C_N}{r^{N-1}} & \text{for } 0<r< R_A,\\[0.3cm]
-\dfrac{A}{r^{1/N}}& \text{for }  r\geq R_A.
\end{cases}
\end{equation*}
Using this bound and \eqref{ineq_T}, $\frac{dM_3(t)}{dt}$ becomes
\begin{equation*}
\begin{split}
\frac{dM_3(t)}{dt} \leq \frac{3}{2}\int_{\mathbb{R}^N} \rho(t,x) &\Big[ \underbrace{\int_{B(x, R_A)} \frac{C_N\, \rho(t,y)}{|x-y|^{N-2}}(|x|+|y|) dy}_{T_r^x} -  \\
&\underbrace{\int_{\mathbb{R}^N\setminus B(x, R_A)} \rho(t,y)
\frac{A}{2}|x-y|^{(N-1)/N} (|x|+|y|) dy}_{T_a^x}\Big] dx\,.
\end{split}
\end{equation*}

Similar to Section 2, we again claim that $T_r^x \leq T_a^x/2$ for
$|x|>2R_A$. We start with controlling $T_r^x$. It is easy to check
that
$$
T_r^x \leq (2|x|+R_A) \int_{B(x, R_A)}  \frac{C_N \,
\rho(t,y)}{|x-y|^{N-2}} dy.
$$
Note that the singularity of the Newtonian kernel is more
difficult to treat than in Section 2. We compensate this
difficulty by using the fact that $\rho(t,x)$ is uniformly bounded by $M_0$
from Lemma \ref{lemma:u_infty}. Hence for $N\geq 2$, we are able
to apply Lemma \ref{lemma:integral} to $\rho(t,\cdot)$, and obtain
\begin{align}
T_r^x &\leq C_3 (2|x|+R_A) \left(\int_{B(x, R_A)} \rho(t,y) dy\,
\right)^{2/N}\label{ineq:t_r_newton}
\end{align}
here $C_3$ only depends on $N$ and $M_0$ as obtained in Lemma
\ref{lemma:u_infty}.

To simplify notation, from now on, we define by $m(t,x)$ the mass
of $\rho$ within radius $R_A$ of $x$ at time $t$, namely
$$
m(t,x) := \int_{B(x, R_A)} \rho(t,y) dy.
$$
Recall that in the beginning of this section we assume that $\rho_0$
integrates to 1, which implies that $\rho(t, \cdot)$ also integrates to 1 for all $t\geq 0$. Then, for any $t$ and $x$, one of the two following
scenarios must be true: either $m(t,x) \geq \frac{1}{2}$, or
$1-m(t,x) > \frac{1}{2}$.

If $m(t,x) \geq \frac{1}{2}$ at some $|x|\geq 2R_A$, it follows
that $m(t,x)^{2/N}$ is comparable to $m(t,x)$. Hence, using
\eqref{ineq:t_r_newton}, we get
\begin{equation*}
\begin{split}
T_r^x&\leq C_3 \,2^{1-2/N} ~m(t,x)  (2|x|+R_A)\\
&\leq 8 C_3\int_{B(x, R_A)} \rho(t,y) (|x|+|y|) dy \quad
(\text{since } |x| \geq 2R_A)\,.
\end{split}
\end{equation*}
Hence by repeating the same argument on the center of mass as in
Section 2 (see \eqref{center_mass_ineq} and the paragraph after
it), we can  choose $A$ to be sufficiently large, then we would
obtain that $T_r^x \leq T_a^x/2$.

On the other hand, if the opposite scenario is true at some
$|x|\geq 2R_A$, i.e.
$$
\int_{\mathbb{R}^N \setminus B(x,R_A)} \rho(t,y) dy >
\frac{1}{2}\,,
$$
then one can directly bound $T_r^x$ by $C(W,N) |x|$ by applying
Lemma \ref{lemma:u_infty} and Lemma \ref{lemma:integral}.
Meanwhile it follows directly from the definition of $T_a^x$ that
$T_a^x \geq \frac{A}{4} |x|$, hence by choosing $A$ sufficiently
large we obtain that $T_r \leq T_a/2$.

Finally, we choose $A$ to be the maximum value needed in the two
scenarios. As a result, $T_r \leq T_a/2$ holds for for all $|x|
\geq 2R_A$, implying that
\begin{equation}
\frac{dM_3(t)}{dt} \leq C_4 - C_5 \int_{\mathbb{R}^N \setminus
B(0, 2R_A)} \rho(t,x) |x| \int_{\mathbb{R}^N \setminus B(x, R_A)}
\rho(t,y) dydx, \label{ineq:dm3dt_newton}
\end{equation}
where $C_4, C_5$ only depends on $W_a$, $N$ and
$\|\rho_0\|_\infty$. Note that this inequality is parallel to the
inequality \eqref{dm_3} for the discrete case.

\subsection{Coupling with the evolution of the support}
Next we will proceed similarly as in Section \ref{couple_discrete}, where most of the
arguments are parallel. We will quickly go through the similar
parts in the proof, and emphasize the differences caused by the
Newtonian repulsive kernel.

At time $t$, we can find $x_0\in\partial\text{supp}(\rho_0)$
depending on $t$, such that
$$
|X_t(x_0)| = R(t) \qquad \mbox{and} \qquad  \frac{d R(t)}{dt} \leq
-(\nabla W\ast \rho)(X_t(x_0), t)\cdot \frac{X_t(x_0)}{R(t)} \,,
$$
similarly to \cite{BCL,BalagueCarrillo}.

Due to \textbf{(W-CONF)}, for any large constant $K_3$ to be determined later, there
exists some radius $R_{K_3} > 6R_A$ such that $w'(r) > K_3/
r^{1/N}$ for all $r > R_{K_3}$. Hence whenever $R(t) > R_{K_3}$,
the growth of $R(t)$ is now controlled by
\begin{equation}
\begin{split}
\frac{d R(t)}{dt} \leq  \int_{B(X_t(x_0), R_A)} \frac{C_N
\,\rho(t,y)}{|y-X_t(x_0)|^{N-1}} dy - \frac{K_3}{12 R(t)^{1/N}},
\label{ineq:drdt_newton_1}
\end{split}
\end{equation}

where the second term on the right hand side is obtained in the
same way as the last term in \eqref{dr_2}, except that the power
$1$ is replaced by $1/N$ due to \textbf{(W-CONF)}.

To deal with the singularity in the first term on the right hand
side, recall that  $\|\rho(t,\cdot) \|_\infty$ is bounded above by
$M_0$ for all time due to Lemma \ref{lemma:u_infty}, which again
enables us to apply Lemma \ref{lemma:integral} to obtain
\begin{equation}
\begin{split}
\frac{d R(t)}{dt} \leq C_6 m(t,X_t(x_0))^{1/N}- \frac{K_3}{12
R(t)^{1/N}}, \label{ineq:drdt_newton}
\end{split}
\end{equation}
where $C_6$ only depends on $N$, $W_a$ and $\|\rho_0\|_\infty$.

Similar to Section \ref{couple_discrete}, we can find some time interval $[t_1, t_2]$,
such that $R(t)$ increases from $R_{K_3} $ to $A_2 R_{K_3} $
within $[t_1, t_2]$, and $\dot R(t_2) > 0$. Here $A_2$ is a
sufficiently large number to be determined at the end of this
subsection. Then we have
$$
\int_{t_1}^{t_2} \frac{dR(t)}{dt} R(t)^{1/N}  > 0,
$$
implying that
\begin{equation}
\int_{t_1}^{t_2} \left( C_6~ m\big(t,X_t(x_0)\big)^{1/N}
R(t)^{1/N} - \frac{K_3}{12} \right) dt> 0. \label{ineq:1/d}
\end{equation}
We apply H\"older's inequality on \eqref{ineq:1/d}, and obtain that
\begin{equation}
\int_{t_1}^{t_2}  m\big(t,X_t(x_0) \big)R(t) dt>
\left(\frac{K_3}{12 C_6}\right)^N (t_2-t_1). \label{ineq:1/d_holder}
\end{equation}
Note that this extra step is needed here but unnecessary in
Section 2, due to the different powers in \textbf{(NL-CONF)} and
\textbf{(W-CONF)}.

Now we are ready to couple the growth of $M_3$ with
\eqref{ineq:1/d_holder}. Since $R(t) > 6 R_A$ for all $t\in[t_1,
t_2]$ (recall that when defining $R_{K_3}$ we set it to be greater
than $6R_A$), we could treat \eqref{ineq:dm3dt_newton} in the same
way as we did in \eqref{dm_4} and \eqref{dm_final}, and bound the
growth of $M_3$ as follows:
\begin{equation}
\frac{dM_3(t)}{dt} \leq (C_4+ \frac{1}{3}C_5 R_A) - C_5
m\big(t,X_t(x_0)\big) R(t) \label{dm_final_newton}.
\end{equation}
Then we integrate \eqref{dm_final_newton} in $[t_1, t_2]$, and it becomes
$$
\int_{t_1}^{t_2}\Big(C_5 m\big(t, X_t(x_0)\big)  R(t) - (C_4 +
\frac{1}{3}C_5 R_A) \Big) dt \leq M_3(t_1) - M_3(t_2).
$$
By putting the above inequality together with \eqref{ineq:1/d_holder}, we can fix $K_3$ to be sufficiently large such that $M_3(t_1) > M_3(t_2)$.

Finally, we prove that if $A_2$ is sufficiently large, we would have
$M_3(t_2) > M_3(t_1)$, hence causing a contradiction. It follows from \eqref{ineq:drdt_newton}
and $\dot R(t_2) > 0$ that
$$
C_6 m\big(t_2,X_{t_2}(x_0)\big)^{1/N} - \frac{K_3}{12
R(t_2)^{1/N}}
> 0,
$$
implying that
\begin{equation*}
M_3(t_2) \geq m\big(t_2,X_{t_2}(x_0) \big) \big(A_2 R_{K_3}
-R_{K_3} \big)^3\gtrsim  A_2^3 R_{K_3}^2,
\end{equation*}
which can be made to be greater than $M_3(t_1)$ if $A_2$ is chosen
to be sufficiently large, thus we obtain a contradiction with
$M_3(t_1) > M_3(t_2)$.  This means that $R(t)$ can never reach
$A_2 R_{K_3}$, thus implies the confinement of support for all
times.

\begin{remark}
Let us emphasize that for potentials $W$ given by
\eqref{W_newton_rep}, we are only able to prove confinement in the
continuum setting, not in the particle setting. The reason is that
in the coupling method we use, we need to bound the repulsion part
of $\frac{d}{dt}R(t)$ using the mass in some neighborhood of the
outermost particle. In the continuum setting this is achieved by
first obtaining an $L^\infty$ bound on $\rho(t,\cdot)$ in Lemma
{\rm\ref{lemma:u_infty}}, then applying Lemma {\rm
\ref{lemma:integral}} to arrive to \eqref{ineq:drdt_newton}.
However, in the particle setting, we are unable to obtain a bound
on the ``local density'' of the particles that is independent of
the particle number, and we are unaware of any such results for
repulsive-attractive kernels to the best of our knowledge.
Intuitively we do expect the ``density'' of particles to be
bounded, since the singular repulsion would not allow the
particles to be densely concentrated.  We find it an interesting
open problem to prove some non-local version of Lemma
{\rm\ref{lemma:u_infty}} for the particle system \eqref{odes} with
$W$ given by \eqref{W_newton_rep}.
\end{remark}

\section{Numerics}
In this section we numerically check the confinement properties of
several potentials together with the long-time behavior of the
corresponding particle systems. Let us remark that in all the
cases we have simulated, for which confinement holds, the long
time behavior of the system seems to converge toward a compactly
supported stationary state. In some of the potentials below, this
has not been rigorously proved. This is an interesting theoretical
question that will be treated elsewhere. Our objective in this
section is to check if the conditions under which confinement has
been shown in previous sections are sharp or not. With this aim,
we remind, as it was said in Section 2, that equation \eqref{pde}
is a gradient flow of the interaction energy
$$
E[\rho]=\frac{1}{2}\int_{\mathbb{R}^N}\int_{\mathbb{R}^N}W(x-y)\,\rho(x)\,
\rho(y)\,dx\,dy
$$
with respect to the Wasserstein distance. Thus, stable stationary
states of \eqref{pde} are local minimizers of the interaction
energy.

In Section 2 we have shown that the radius $R(t)$ defined by
\eqref{def:R} is bounded by a constant $\overline{R}$ that depends
only on the potential $W$ and the initial data. Moreover,
$\overline{R}$ is independent of the number of particles and under
certain additional assumptions, see Section 2, we know that the
particle systems are indeed good approximations of the solutions
to the continuum model \eqref{pde}. For this reason, we have
chosen a particle framework to perform our numerical
investigation. We also follow the idea of decreasing the energy
since stationary states are local minimizers of the energy. Given
$n$ particles located at $x_1$, \dots, $x_n\,\in\,\mathbb{R}^N$
with masses $m_1=m_2=\cdots=m_n=1/n$, their discrete interaction
energy is given by
\begin{equation*} \label{eq:intercenergyparticles}
E[x_1,\dots,x_n]=\frac{1}{2n^2}\sum_{i=1}^{n}\sum_{\substack{j=1\\j\neq i} }^n W(x_i-x_j).
\end{equation*}

The simulations are done by an explicit Euler scheme leading to a
trivial gradient descent method as long as the energy is
decreasing at each time step. This method allows to efficiently
solve for stationary states of \eqref{odes}. In stiffer
situations, as for the Morse potentials below, an explicit
Runge-Kutta method is used instead. These methods are essentially
the same as the ones used in \cite{BUKB,soccerball} for finding
stationary states of different repulsive-attractive potentials.
Our stopping criterion is to achieve a numerical steady state. For
us, a numerical steady state is a particle distribution for which
the discrete $l^\infty$-norm of the velocity field in \eqref{odes}
is below some predetermined threshold, which we impose to be
$0.001/n$.

The section is divided into three subsections, each one showing
the results for a particular chosen potential. The limit growth
for the attractiveness of the potential at infinity under
condition \textbf{(NL-CONF)} is $\log(r)$. For this reason in
Section 4.1 we constructed a piecewise potential with exact
logarithmic attraction at infinity. This selection has been done
to check the sharpness of condition \textbf{(NL-CONF)}. In
Subsection 4.2 we go further to take a piecewise potential growing
at infinity exactly like $\log(\log(r))$, which grows even slower
at infinity compared to $\log(r)$. This potential does not satisfy
the condition \textbf{(NL-CONF)}. At the end of this section, in
Subsection 4.3, we also analyze the case of the Morse potential.
This potential is known to be a repulsive-attractive potential
under certain choices of the parameters with negligible attractive
strength at infinity, i.e., $W(x)\to 0$ as $|x|\to \infty$. These
potentials are more interesting in terms of biological relevance
as discussed in \cite{Dorsogna,CMP}.

As a final remark, we point out that all the used potentials are
not singular at the origin and simulations are performed in
dimension $N=2$.


\subsection{Logarithmic attraction at infinity}
We show confinement when the potential has exact logarithmic
attraction at infinity. The chosen potential is
\begin{equation}\label{eq:potlog}
w(r)=
\begin{cases}
{\frac {95}{2}}\,{r}^{2}-{\frac {83}{6}}\,r-64\,{r}^{3}+{\frac {239}{6}}\,{r}^{4}-\frac{19}{2}\,{r}^{5}&0\leq r\leq 1, \\ \log  \left( r \right) &r>1.
\end{cases}
\end{equation}
It can be checked that it is a repulsive-attractive satisfying
$w(0)=w(1)=0$, $w(r)\in C^3(0,+\infty)$, and the repulsion at the
origin is $\simeq -r$. The stationary states are shown in
Table~\ref{tab:tlog}.
\begin{table}[h]
\centering
\begin{tabular}{||c|c|c|c||}
\hline $n=100$ & $n=200$ & $n=500$ & $n=1000$\\\hline
\scalebox{0.23}{\includegraphics{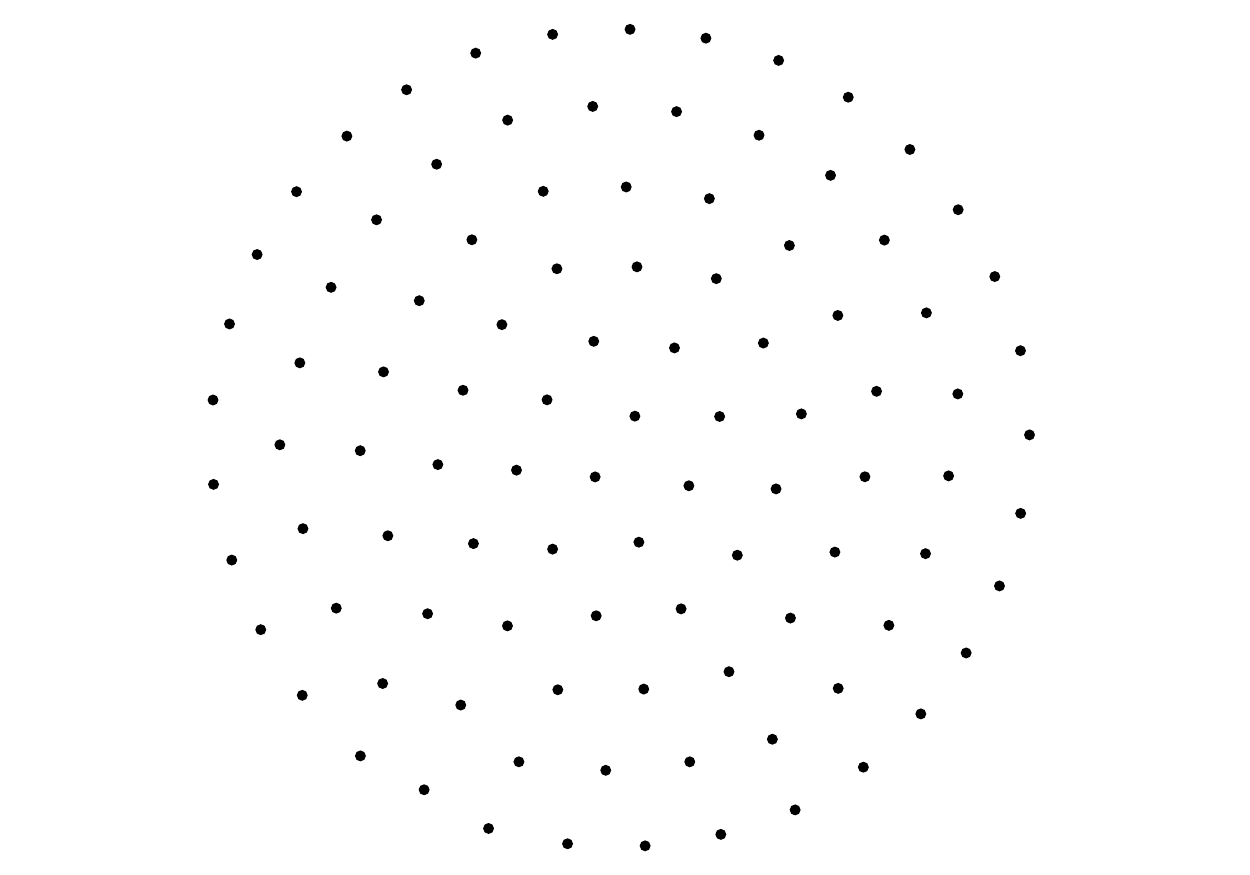}} &
\scalebox{0.23}{\includegraphics{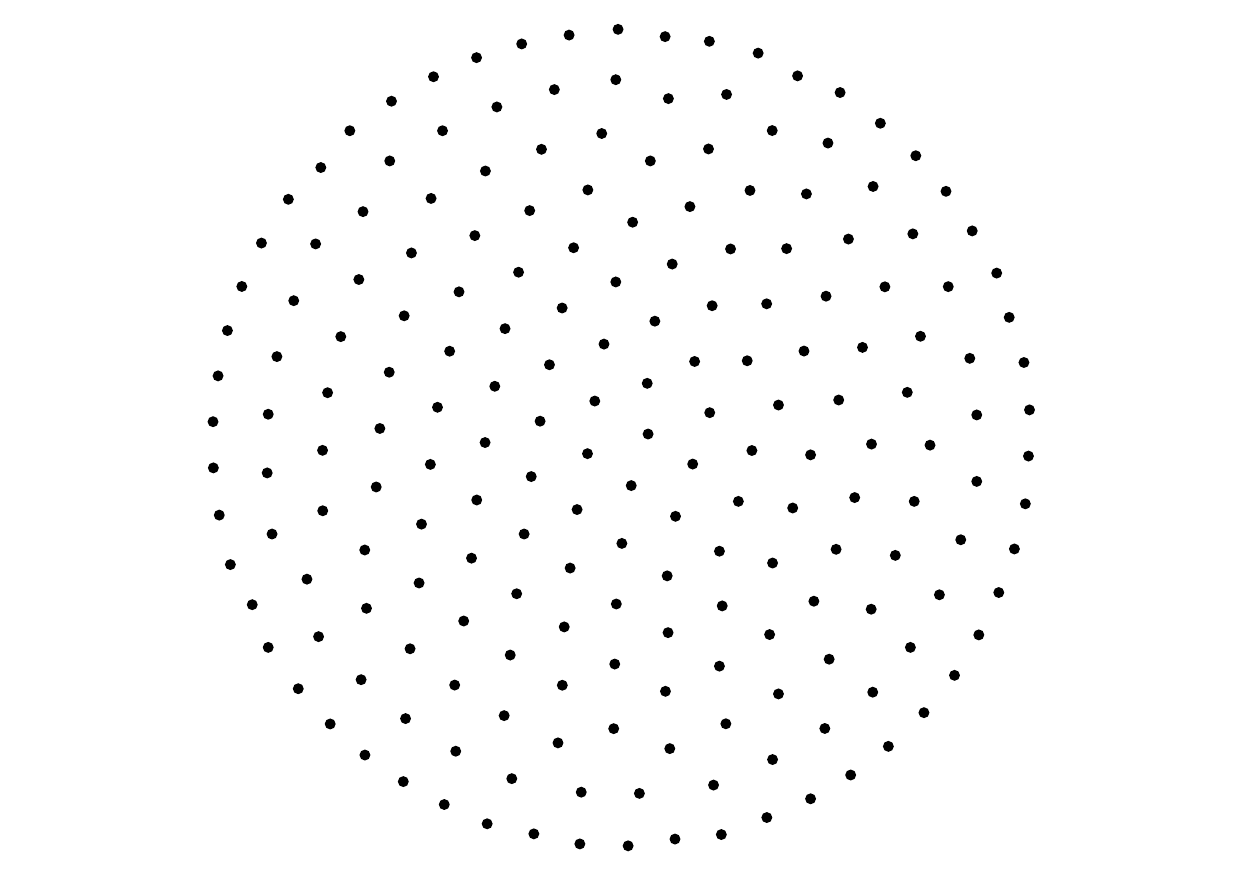}} &
\scalebox{0.23}{\includegraphics{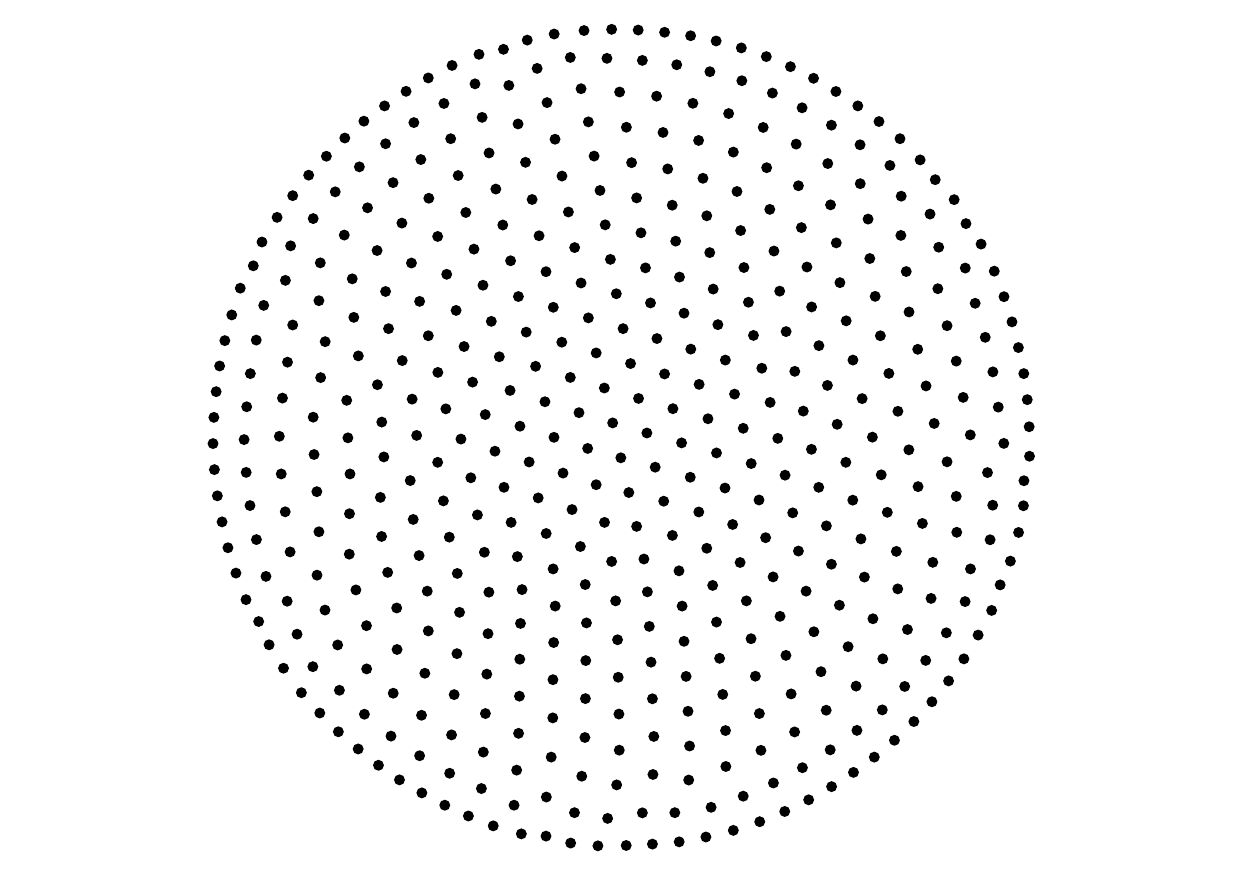}} &
\scalebox{0.23}{\includegraphics{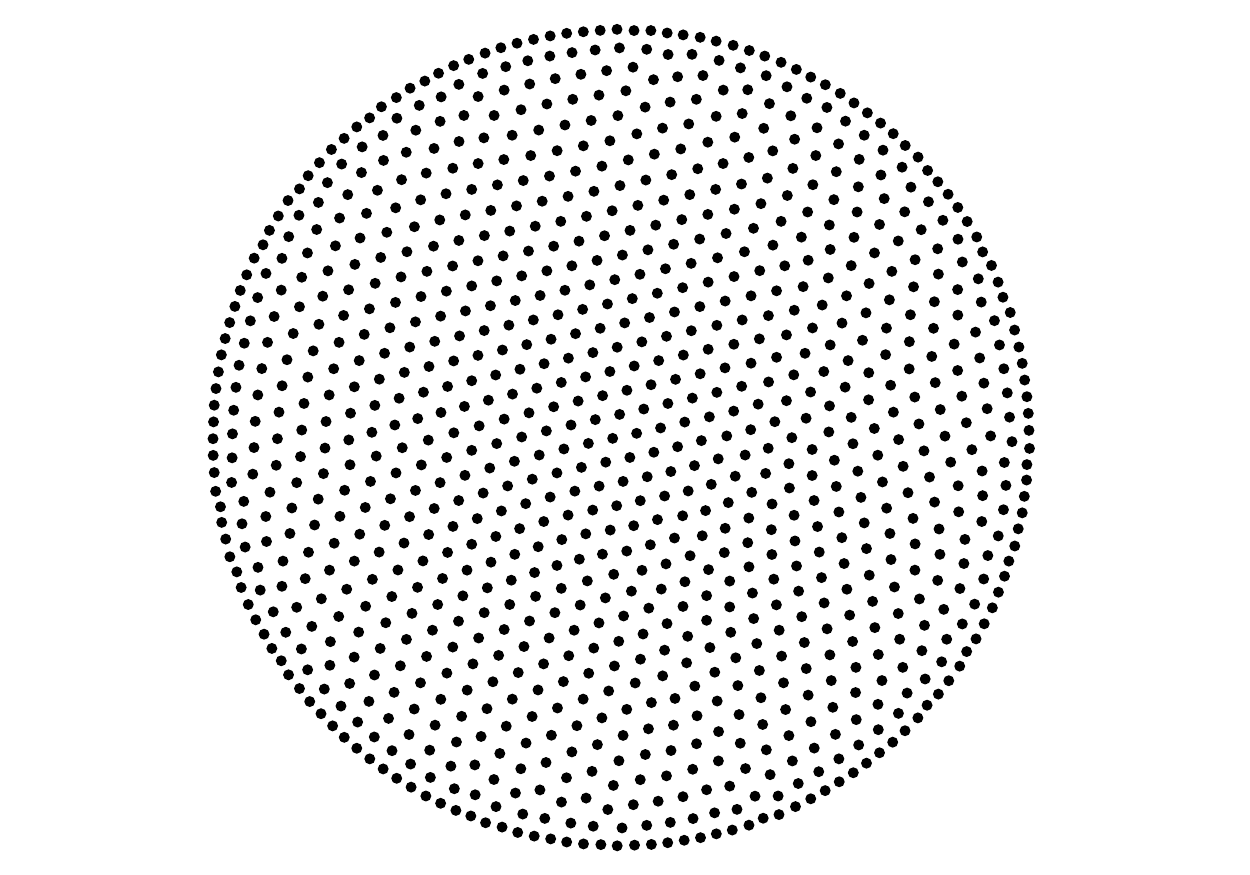}} \\\hline
$r\sim0.2286$ & $r\sim0.2315$ & $r\sim0.2343$ & $r\sim 0.2352$
\\\hline
\end{tabular}
\caption{Stationary states and radius of their support as a
function of the number of particles for the potential $w(r)$ given
in equation \eqref{eq:potlog}.} \label{tab:tlog}
\end{table}
We have chosen initial data in such a way that the particles feel
the logarithmic interaction by randomly placing $n$ particles in a
centered square in such a way that $|x_i-x_j|>{\rm 1}$ for some
values of $i,j\in\lbrace1,2,\dots,n\rbrace$. We have run
simulations varying the number of particles $n$ and the initial
data. For each simulation the center of mass $C_n$ for the
particle system was computed and then
\[
   r_n(t)=\max_{1\leq j\leq n}|x_j(t)-C_n|.
\]
It is observed that $r_n(t)<1$ in all the cases for large times
and it converges to some asymptotic value. This fact can be
explained because when all the particles are out of the range of
the $\log(r)$ part then the radius and the behavior depends only
on the polynomial part of the potential. Simulations indicate that
there is confinement for this potential and thus, condition
\textbf{(NL-CONF)} is not sharp. Figure~\ref{fig:evoradlog} shows
the evolution of the radius as a function of the particle number
$n$ and as a function of time for a particular initial data.

\begin{figure}[h]
\centering
\subfigure[]{\includegraphics[scale=0.34]{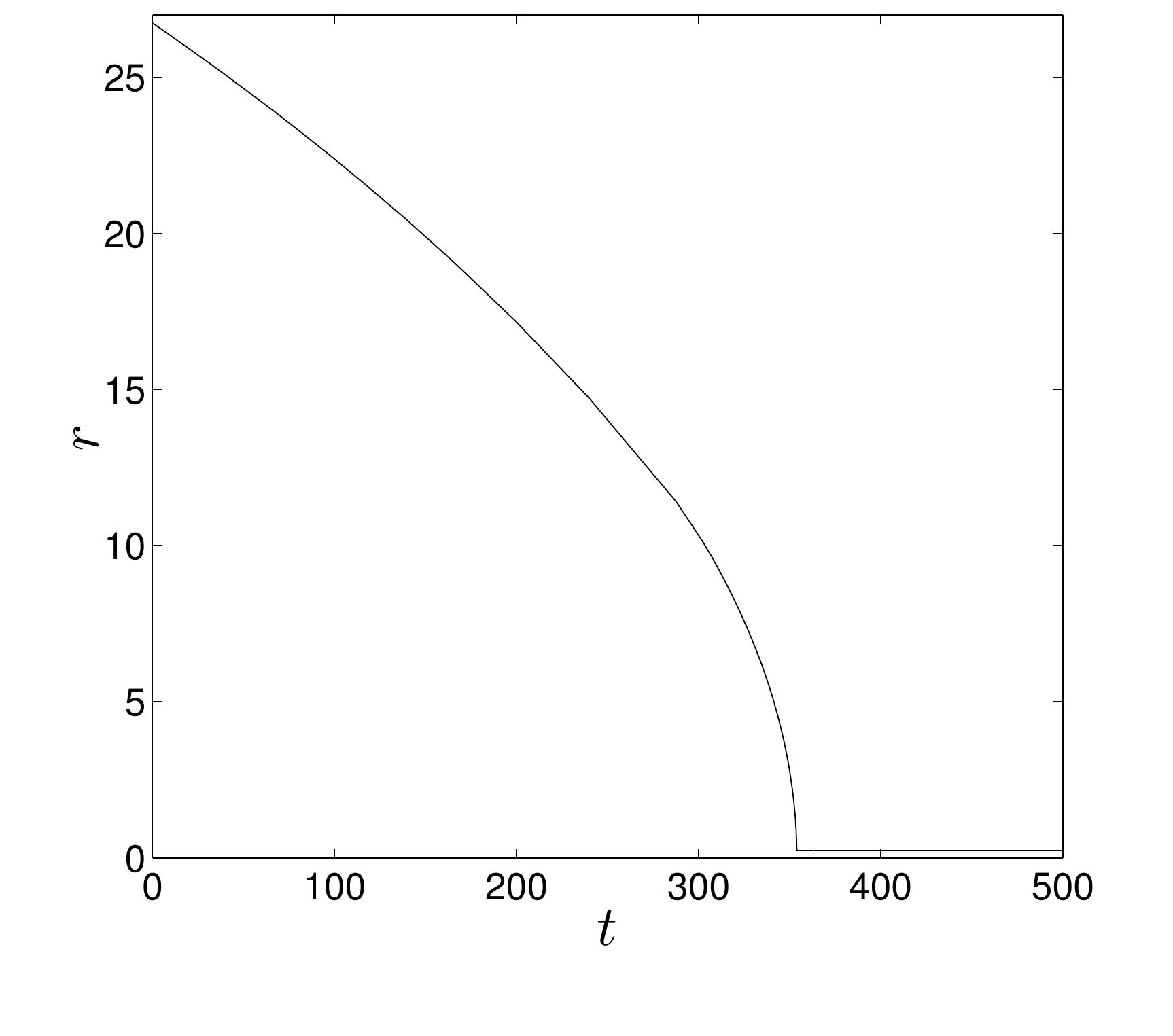}}
\subfigure[]{\includegraphics[scale=0.32]{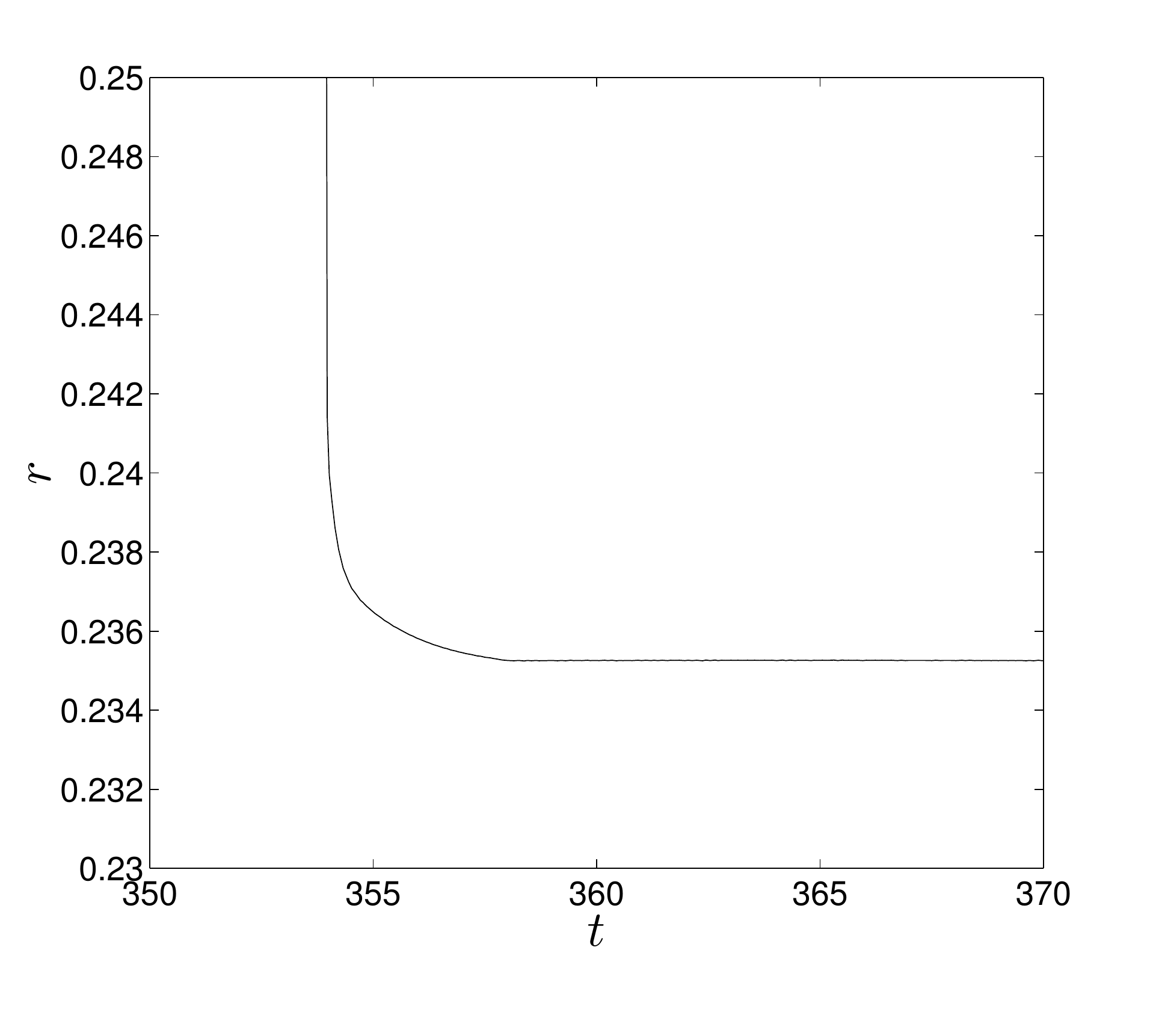}}
\subfigure[]{\includegraphics[scale=0.34]{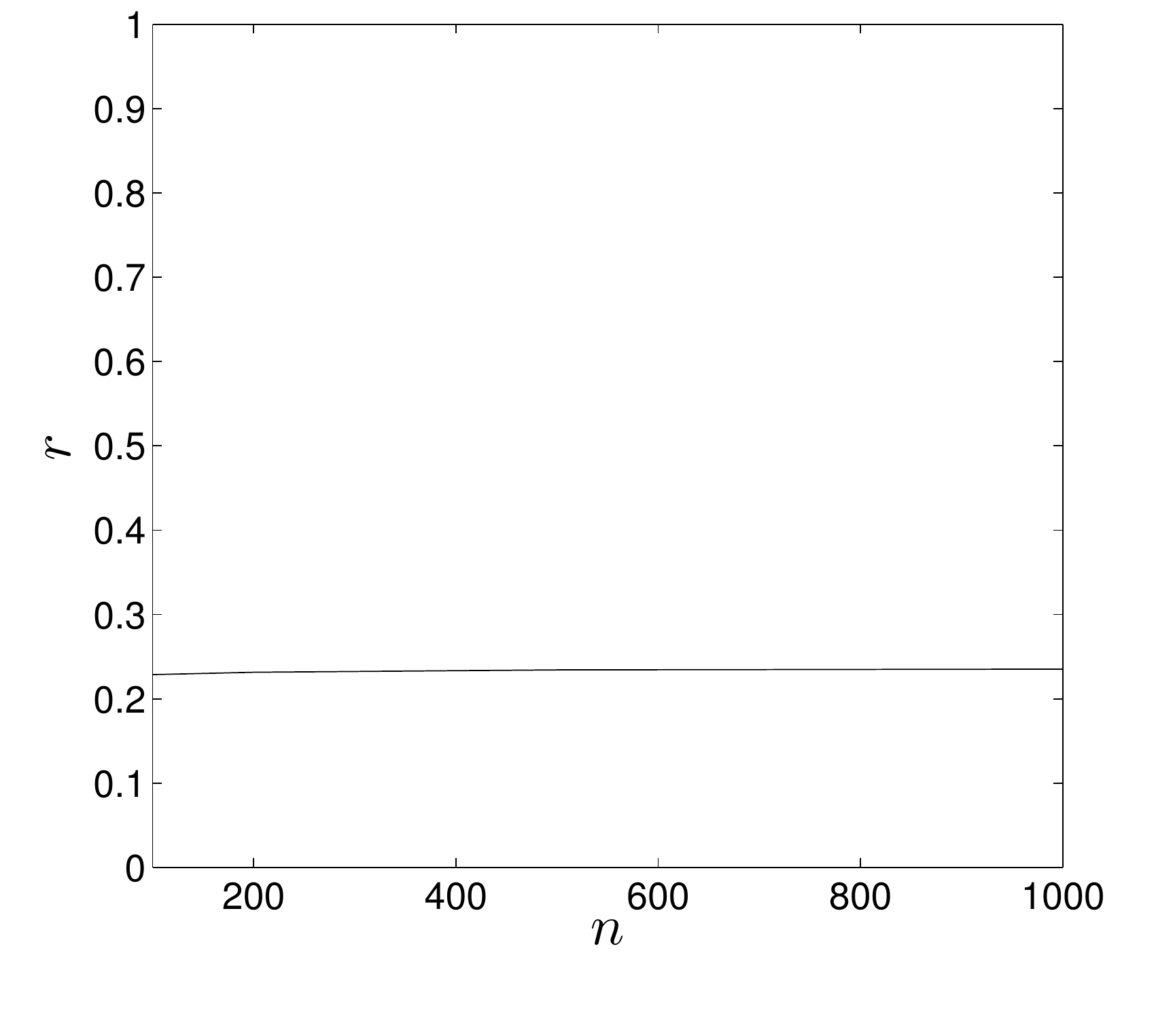}}
\caption{(a) Evolution of the radius as a function of $t$ for
$n=1000$. (b) Zoom of picture (a). (c) Evolution of the radius as
a function of $n$.} \label{fig:evoradlog}
\end{figure}

\subsection{Log-log attraction at infinity}

In this case we consider a potential behaving like $\log(\log(r))$
at infinity
\begin{equation}\label{eq:potloglog}
w(r)=
\begin{cases}
{\frac {r \left( r-{{\rm e}} \right) }{ {{\rm e}}^{2}}}-2\,{\frac {r \left( r-{{\rm e}} \right) ^{2}}{
 {{\rm e}} ^{3}}}+{\frac {19}{6}}\,{\frac {r
 \left( r-{{\rm e}} \right) ^{3}}{ {{\rm e}}^{4}}}
&0\leq r\leq {\rm e}, \\ \log\left(\log  \left( r \right)\right) & r>{\rm e}.
\end{cases}
\end{equation}
This potential satisfies $w(0)=w({\rm e})=0$, it is
$C^3(0,+\infty)$, and the repulsion at the origin is $\simeq -r$.
The numerical experiments suggest that by increasing the number of
particles, the radius of the support of the stationary state
increases and stabilizes. Table~\ref{tab:tloglog} shows the
stationary states as a function of the number of particles $n$ for
this potential. These numerical simulations, together with the
evolution of the radius of the support both in time and as a
function of the number of particles not shown here, indicate that
even if the growth at infinity of the potential is less than
$\log(r)$ there is still confinement.

\begin{table}[ht]
\centering
\begin{tabular}{||c|c|c|c||}
\hline $n=100$ & $n=200$ & $n=500$ & $n=1000$\\\hline
\scalebox{0.23}{\includegraphics{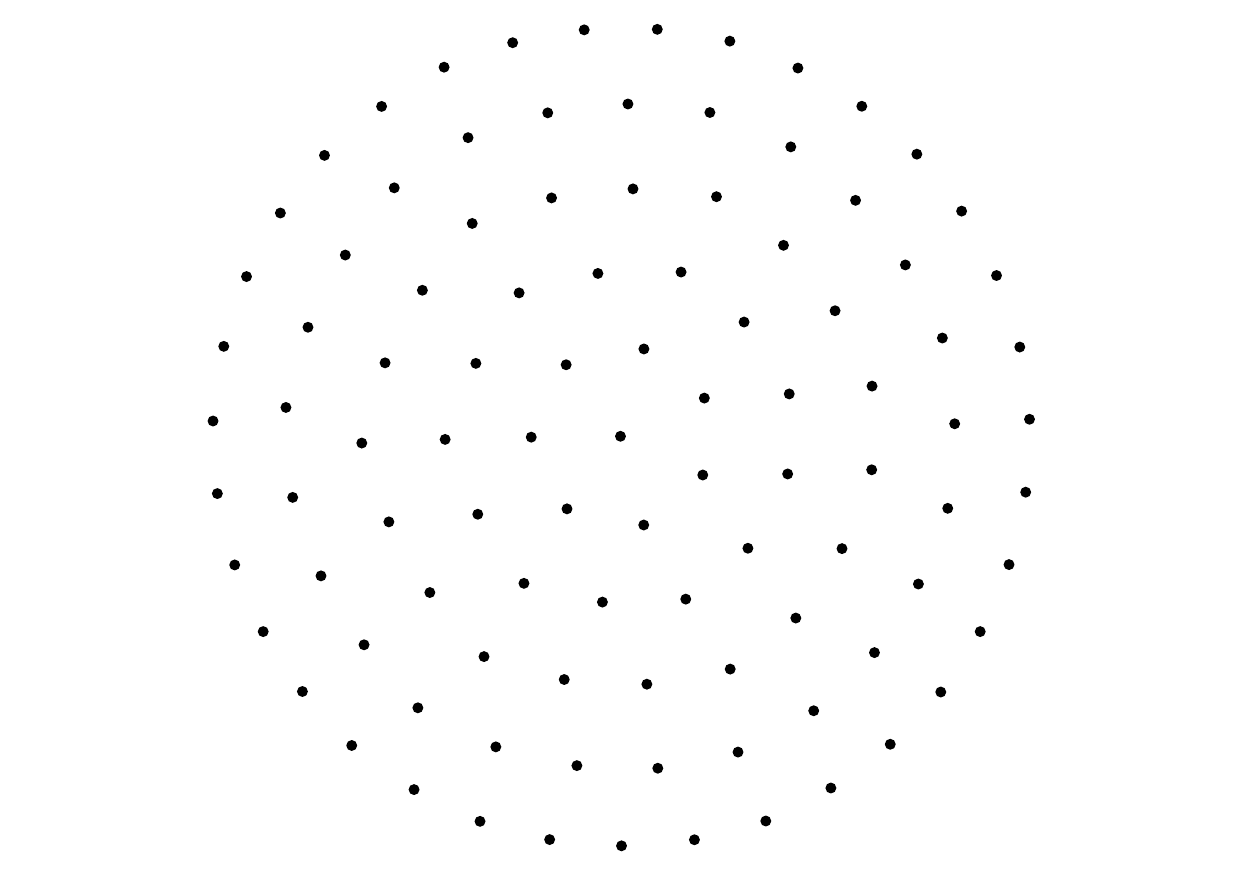}} &
\scalebox{0.23}{\includegraphics{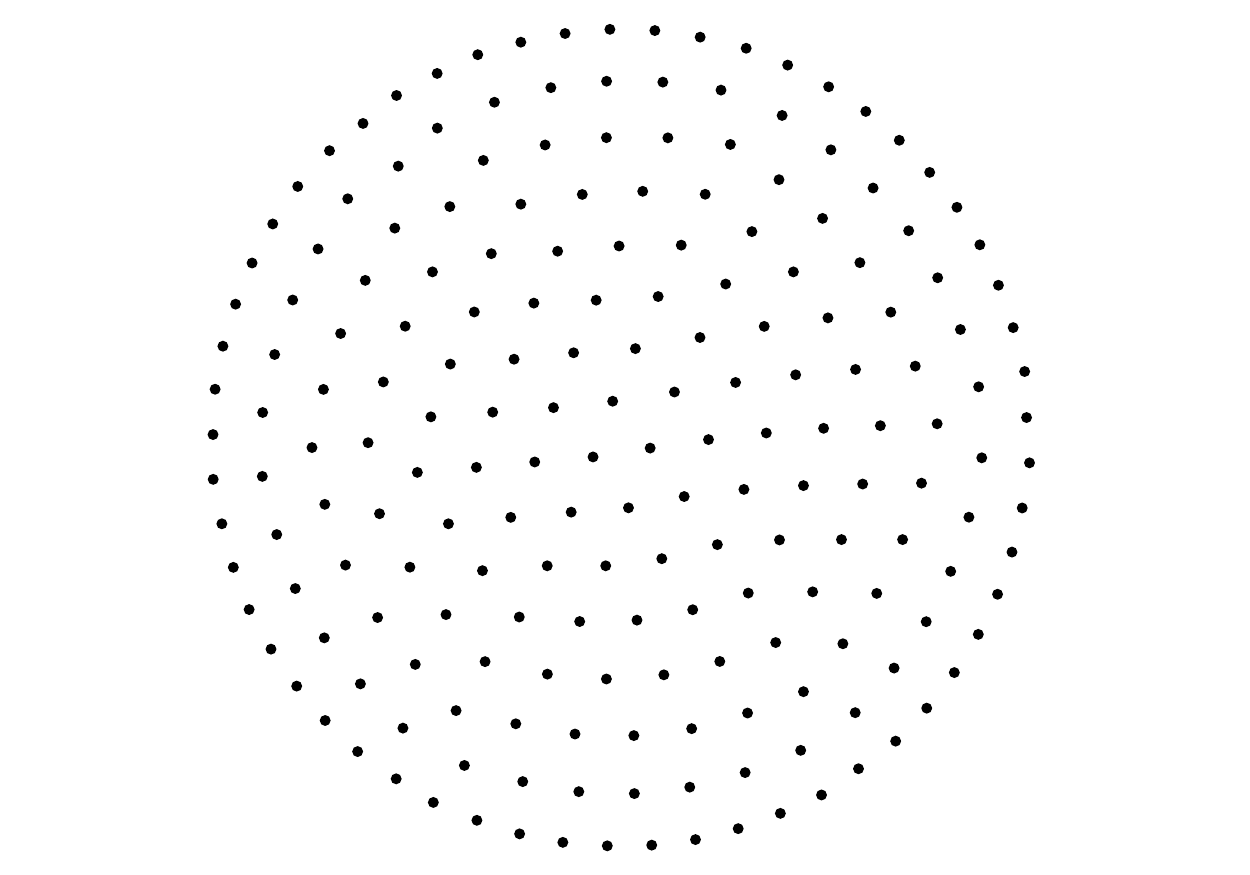}} &
\scalebox{0.23}{\includegraphics{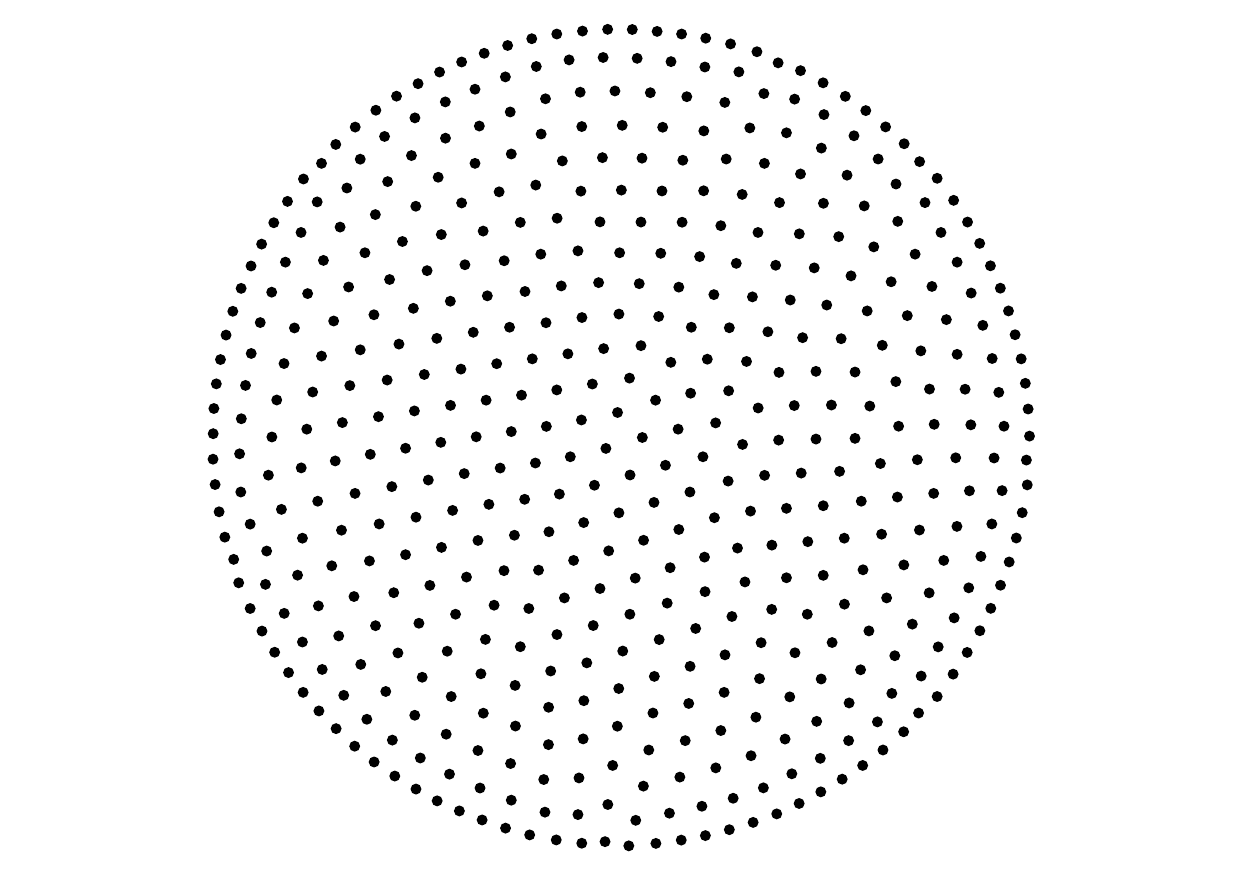}} &
\scalebox{0.23}{\includegraphics{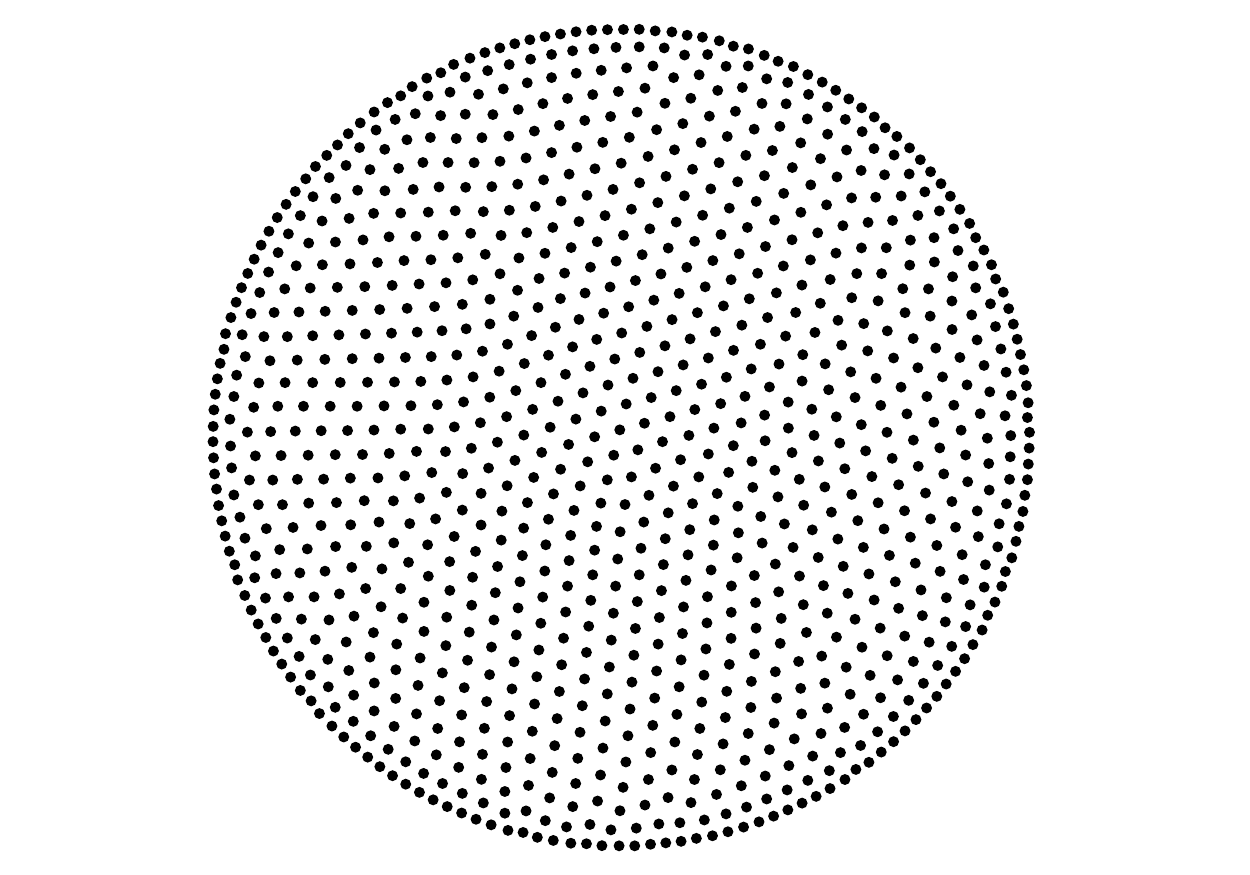}} \\\hline
$r\sim0.7838$ & $r\sim0.7954$ & $r\sim0.8052$ & $r\sim 0.8085$
\\\hline
\end{tabular}
\caption{Stationary states and radius of their support as a
function of the number of particles for the potential $w(r)$ given
in equation \eqref{eq:potloglog}.} \label{tab:tloglog}
\end{table}


\subsection{Morse potential}
The usual form of this potential is the following
\[
    U(r)=-C_Ae^{-r/l_A}+C_Re^{-r/l_R},
\]
where constants $C_A$ and $C_R$ are the attraction and repulsion
strength respectively and the constants $l_A$ and $l_R$ are their
respective length scales. For our simulations we will take the
scaling shown in \cite{Dorsogna,CMP}. That is,
\begin{equation*}\label{eq:morse}
U(r)=C_A(V(r)-C\,V({r/l })),
\end{equation*}
where $V(r)=-\exp(-r/l_A)$ and $C=C_R/C_A$ and $l=l_R/l_A$. It is
known for this potential \cite{Dorsogna,CMP} that for $C>1$ and
$l<1$ the potential $U(r)$ is short-range repulsive and long-range
attractive with a unique minimum defining a typical distance
between particles. Also, in this regime, the condition $Cl^N=1$
distinguishes between the so-called H-stable and catastrophic
regimes.

\noindent{\bf H-stable case:} In our simulations we fix the
parameters as $C_A=l_A=1$, $C_R=1.9$, and $l_R=0.8$ leading to
$C=1.9$, $l=0.8$, and $Cl^2=1.216>1$.

\begin{table}[ht]
\centering
\begin{tabular}{||c|c|c|c||}
\hline $n=100$ & $n=500$ & $n=1000$ & $n=2000$\\\hline
\scalebox{0.23}{\includegraphics{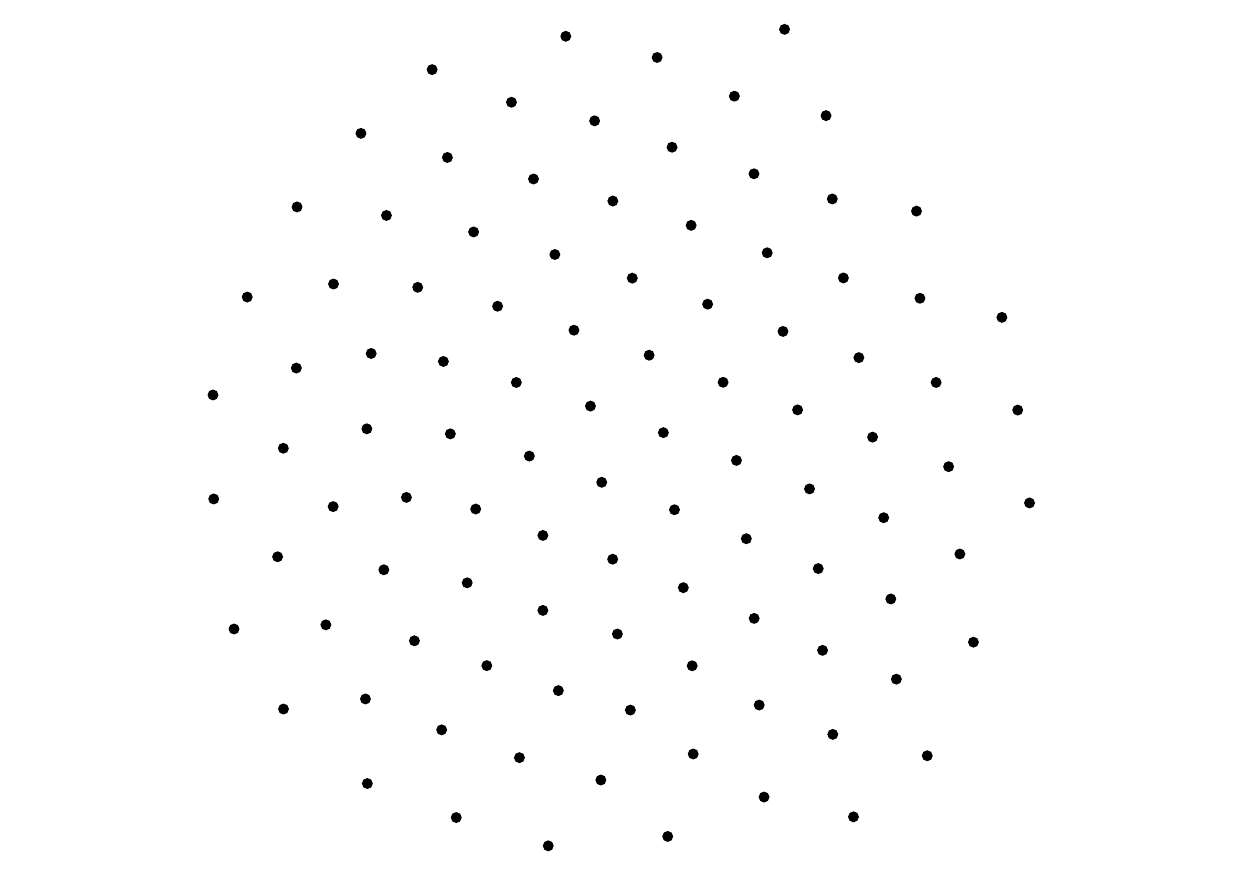}} &
\scalebox{0.23}{\includegraphics{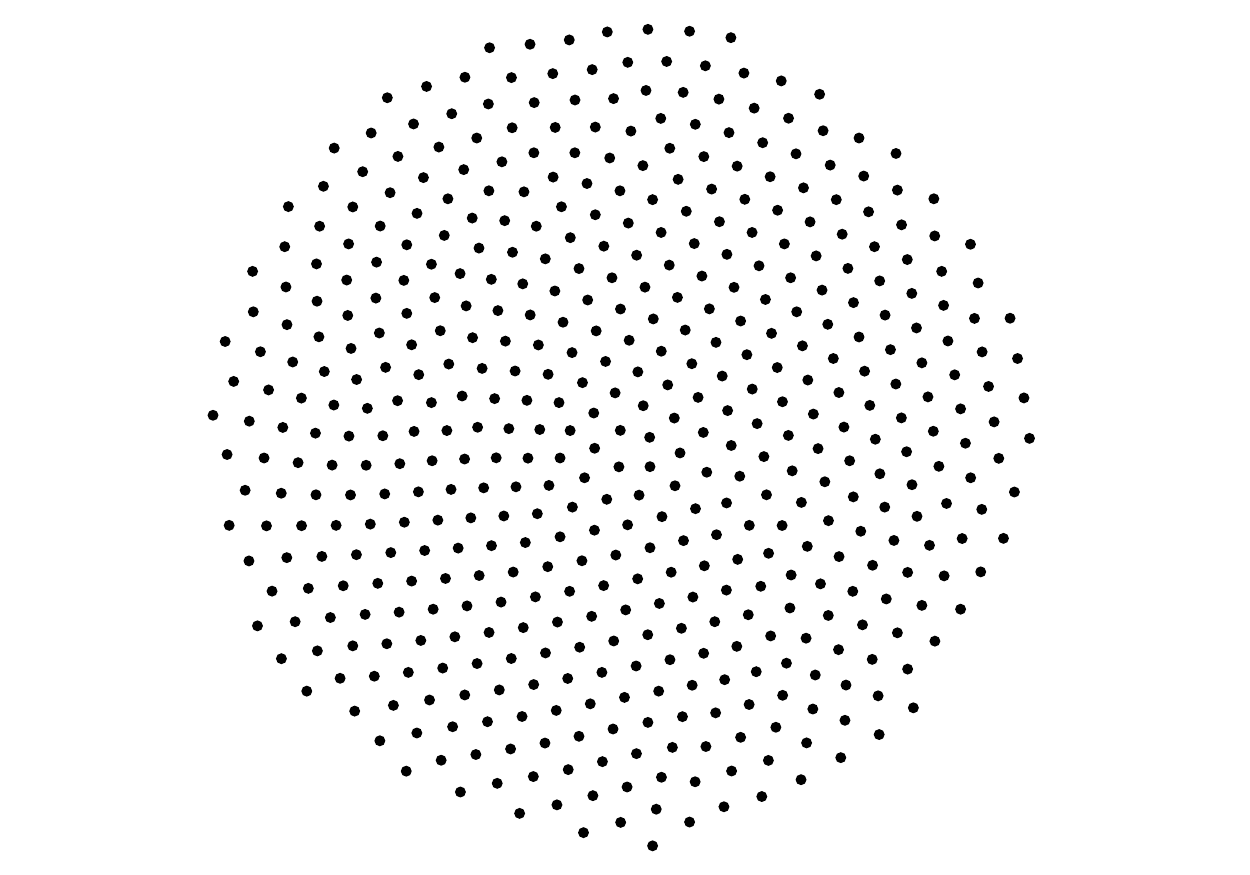}} &
\scalebox{0.23}{\includegraphics{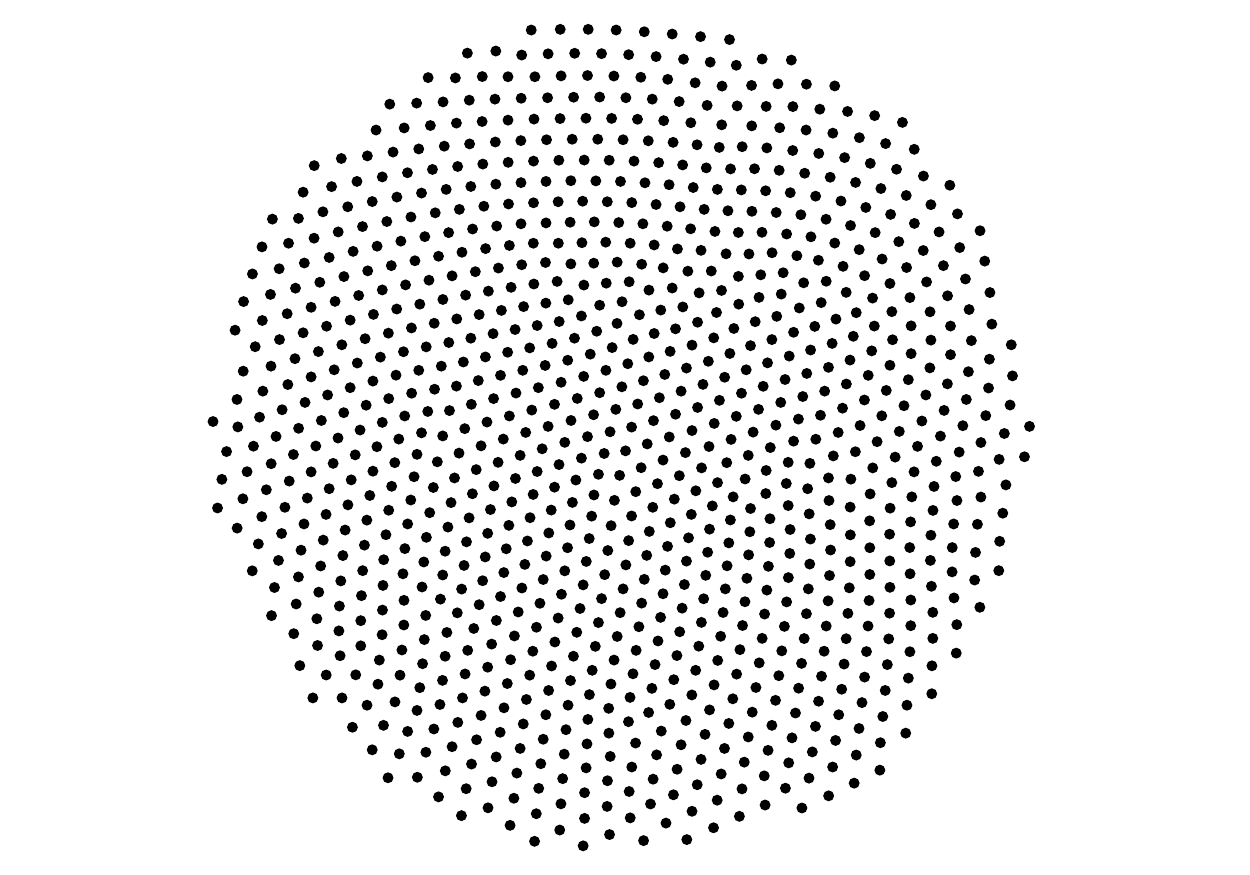}} &
\scalebox{0.23}{\includegraphics{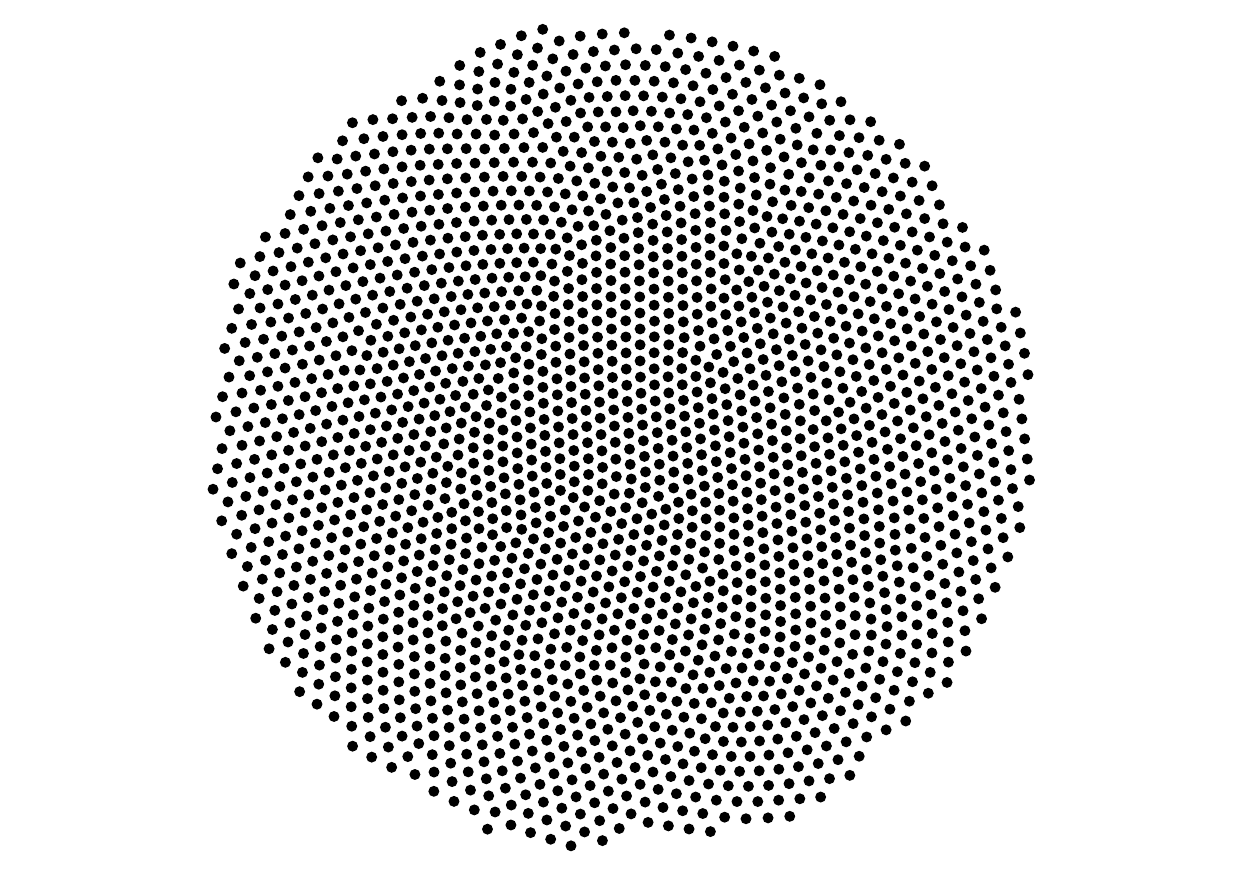}} \\\hline
$r\sim15.96$ & $r\sim 34.002 $ & $r\sim 47.35$ & $r\sim 64.45$
\\\hline
\end{tabular}
\caption{Stationary states and radius of their support as a
function of the number of particles for the potential $U(r)$ given
in equation \eqref{eq:potloglog}.}
\label{tab:morseH}
\end{table}

Numerical experiments in Table~\ref{tab:morseH} demonstrate that
the radius increases by increasing the number of particles, but
with a slower rate. As clearly visualized in
Figure~\ref{fig:radHmorse}(a), the radius appears to grow like a
square root function as the number of particles increases. This
observation is further supported by numerical evidence in
Figure~\ref{fig:radHmorse}(b), where we plot the square of the
radius versus the number of particles, and the linear regression
provides a good fit to the data.

It would be interesting to study the H-stable case in more
details, although it is outside of the scope of this paper. From
these numerical results, we can extract a conjecture that for the
continuum system the support of the density would go unbounded
over time, i.e., the confinement result should not hold for the
H-stable potential.

\begin{figure}
\centering
\subfigure[]{\includegraphics[scale=0.36]{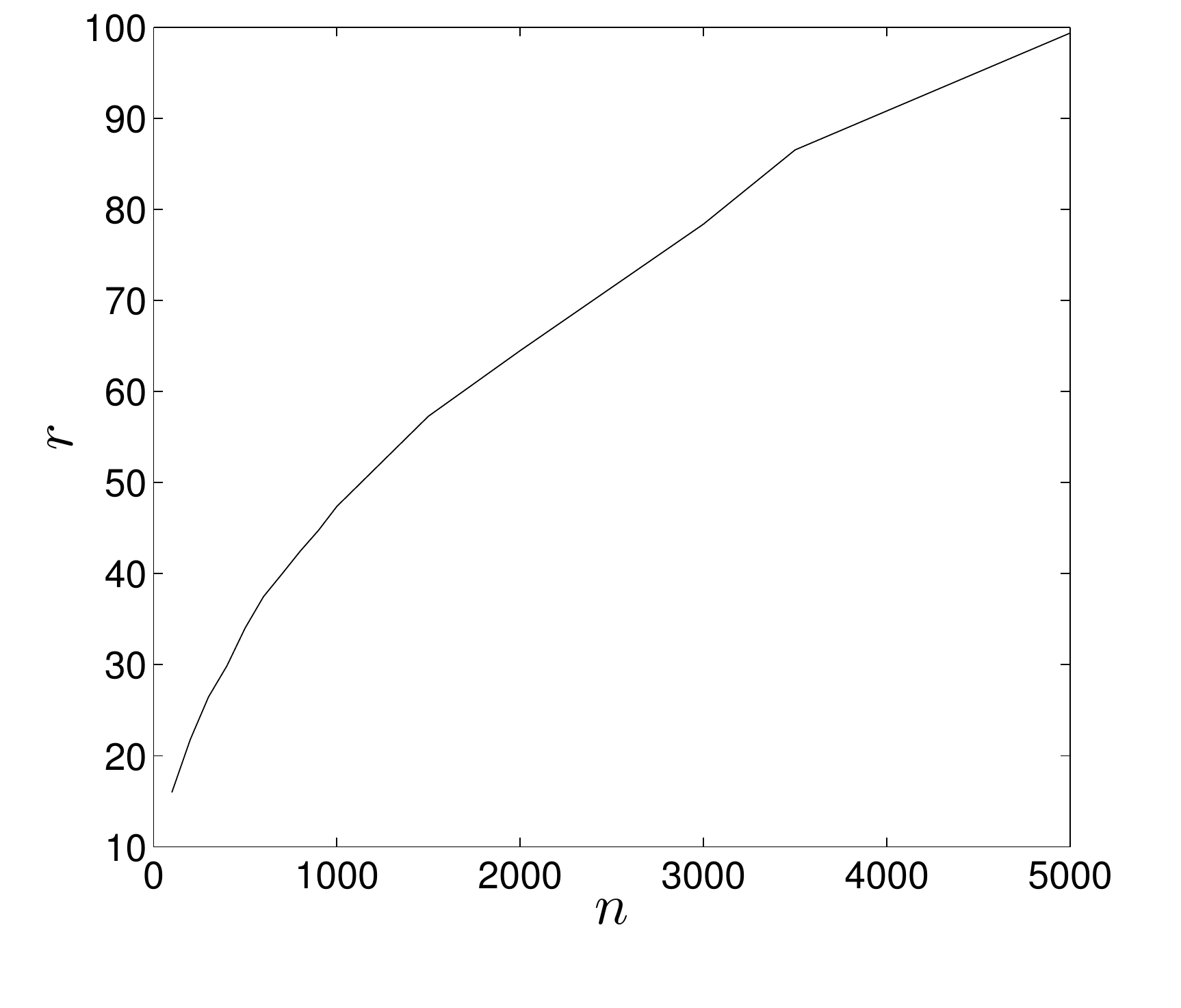}}
\subfigure[]{\includegraphics[scale=0.36]{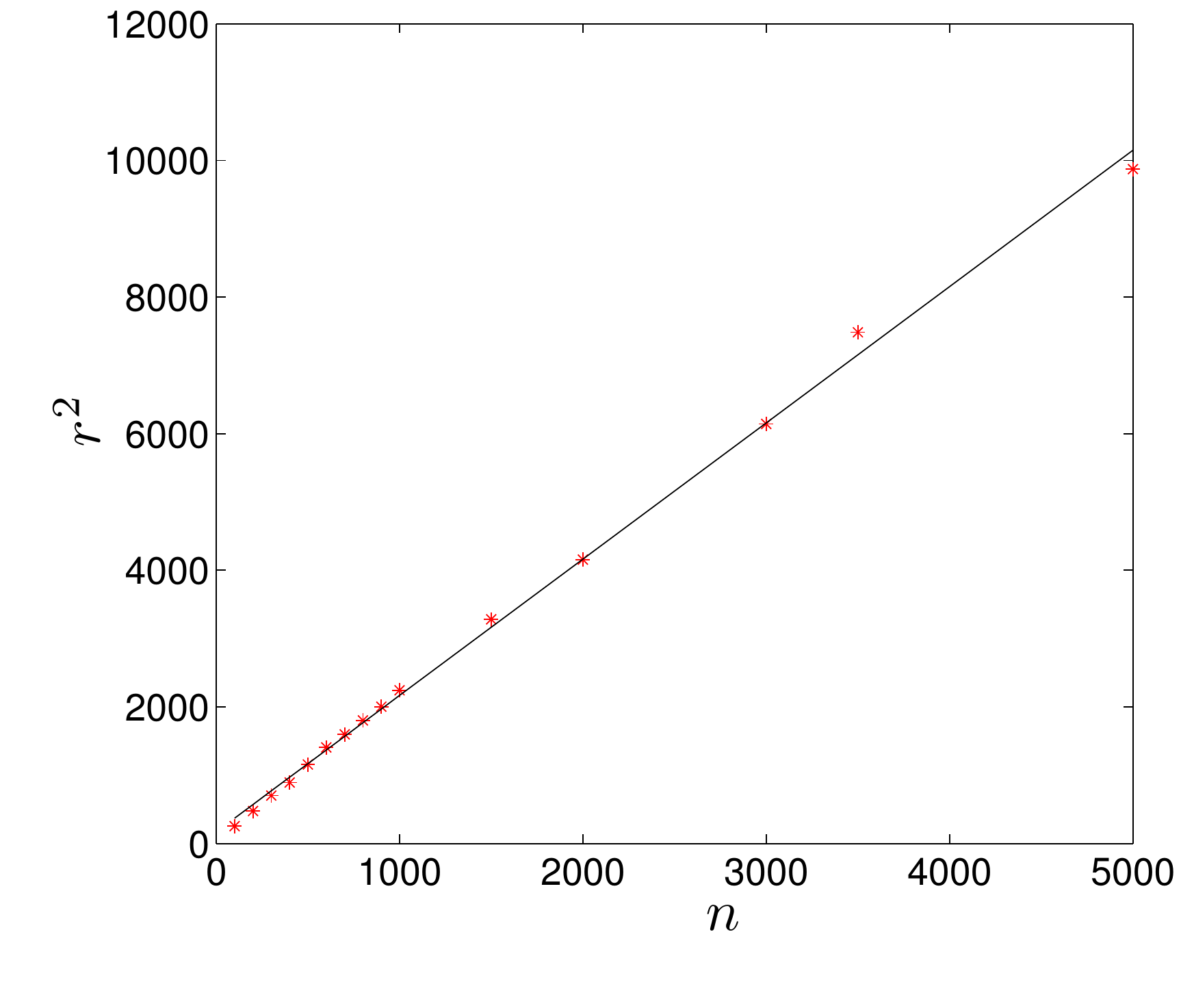}}
\caption{H-stable Case: (a) Evolution of the
radius as a function of $n$. (b) Squared radius of the support of
the steady state as a function of $n$ and the linear regression
curve $y=170.34+1.99n$ computed using the crossed points.}
\label{fig:radHmorse}
\end{figure}

\noindent {\bf Catastrophic case:} The parameters we choose for
the experiments are $C_A=l_A=1$, $C_R=1.3$ and $l_R=0.2$ so that
$Cl^2=0.052<1$.

\begin{table}[h]
\centering
\begin{tabular}{||c|c|c|c||}
\hline $n=100$ & $n=500$ & $n=1000$ & $n=2000$\\\hline
\scalebox{0.23}{\includegraphics{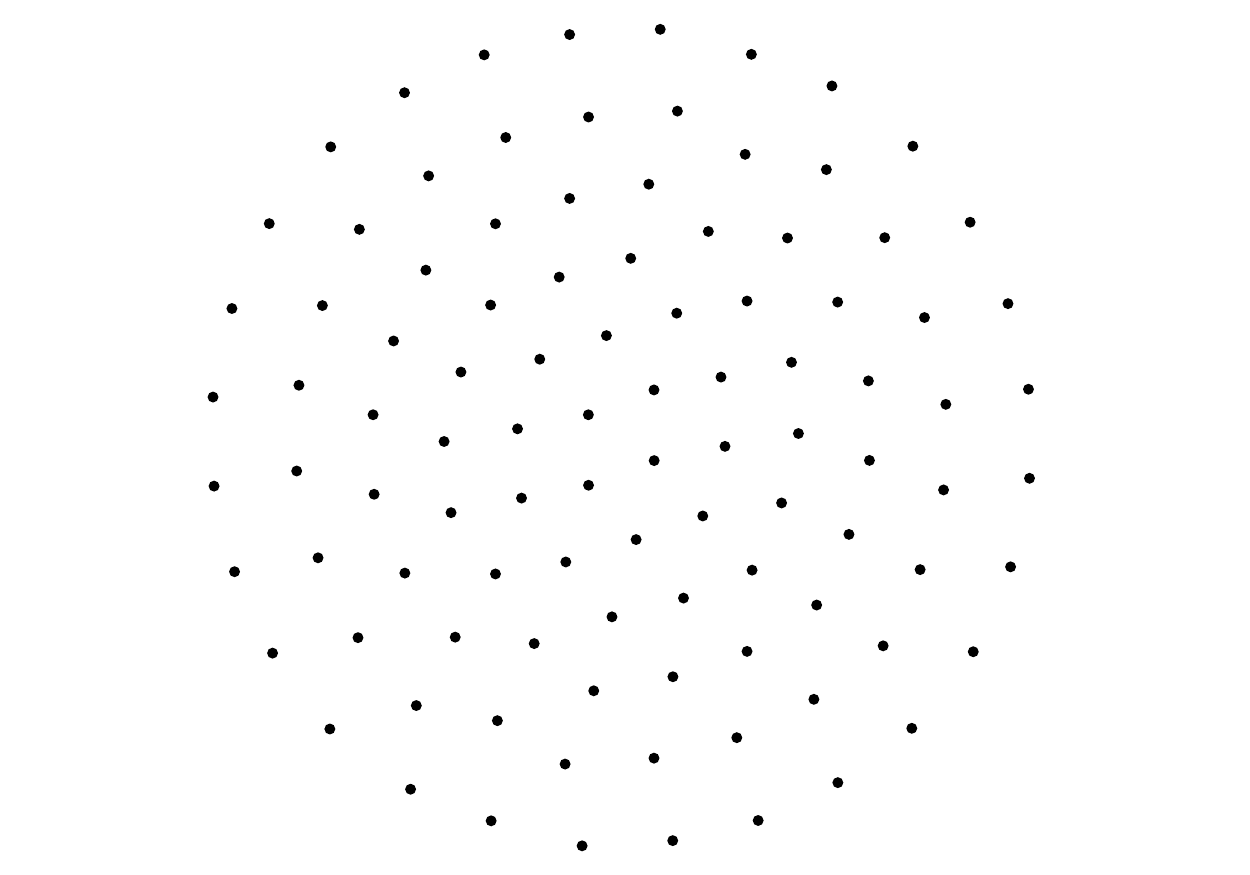}} &
\scalebox{0.23}{\includegraphics{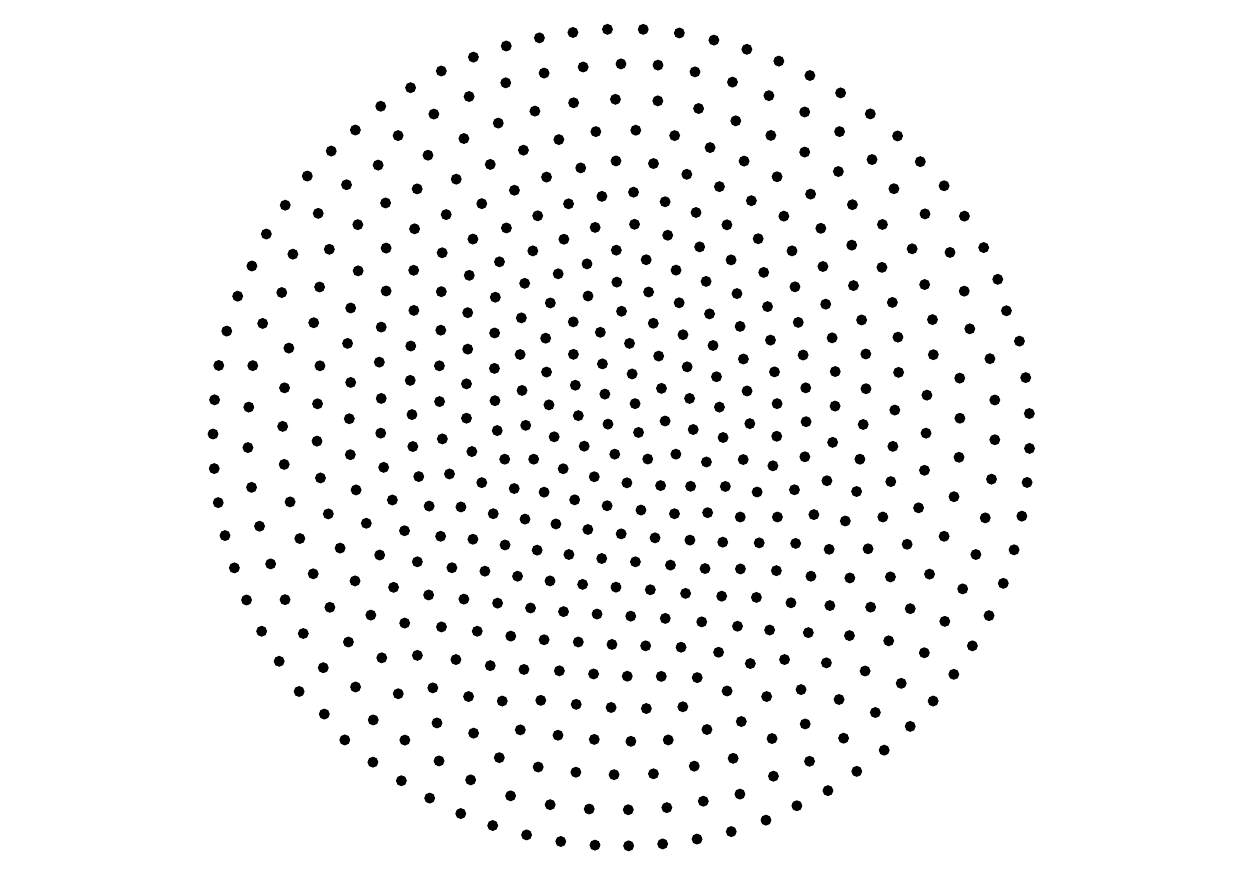}} &
\scalebox{0.23}{\includegraphics{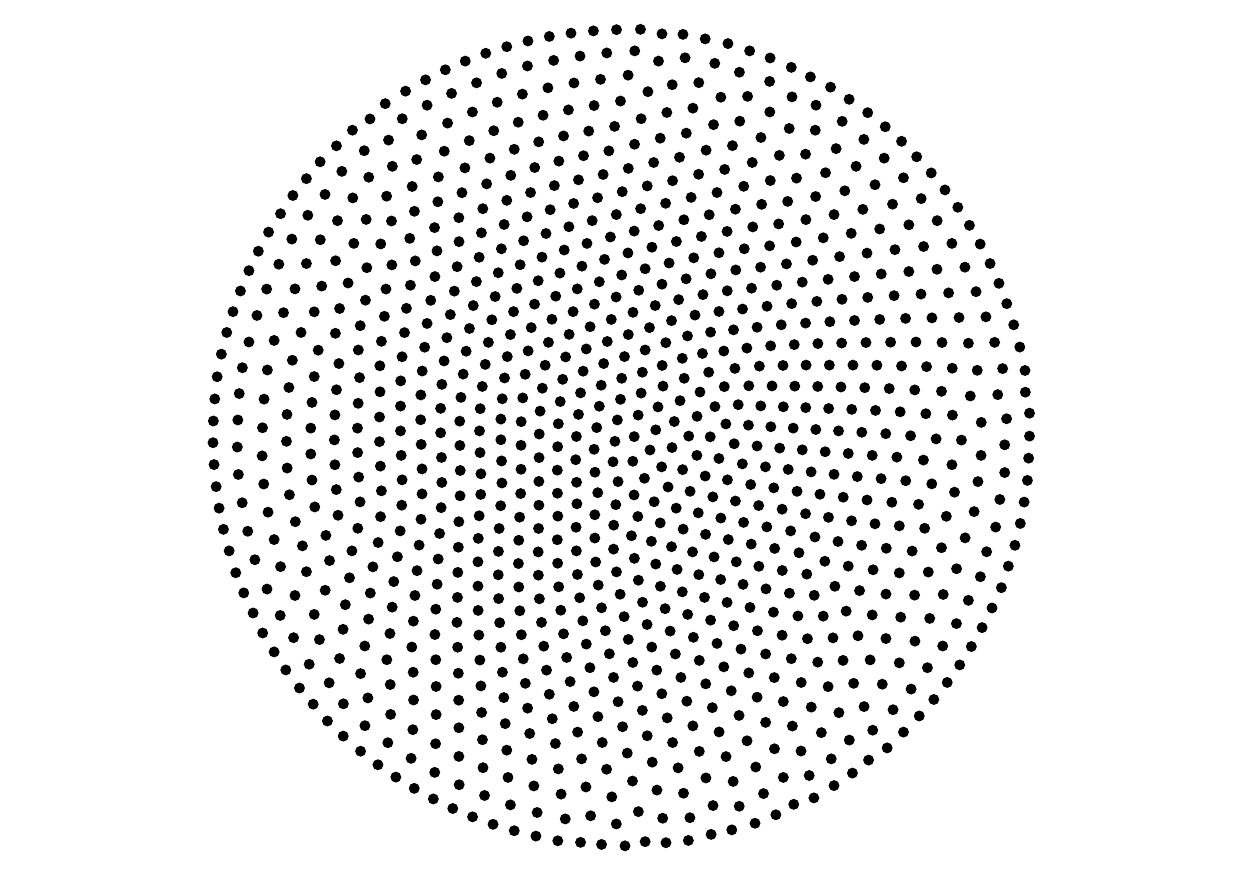}} &
\scalebox{0.23}{\includegraphics{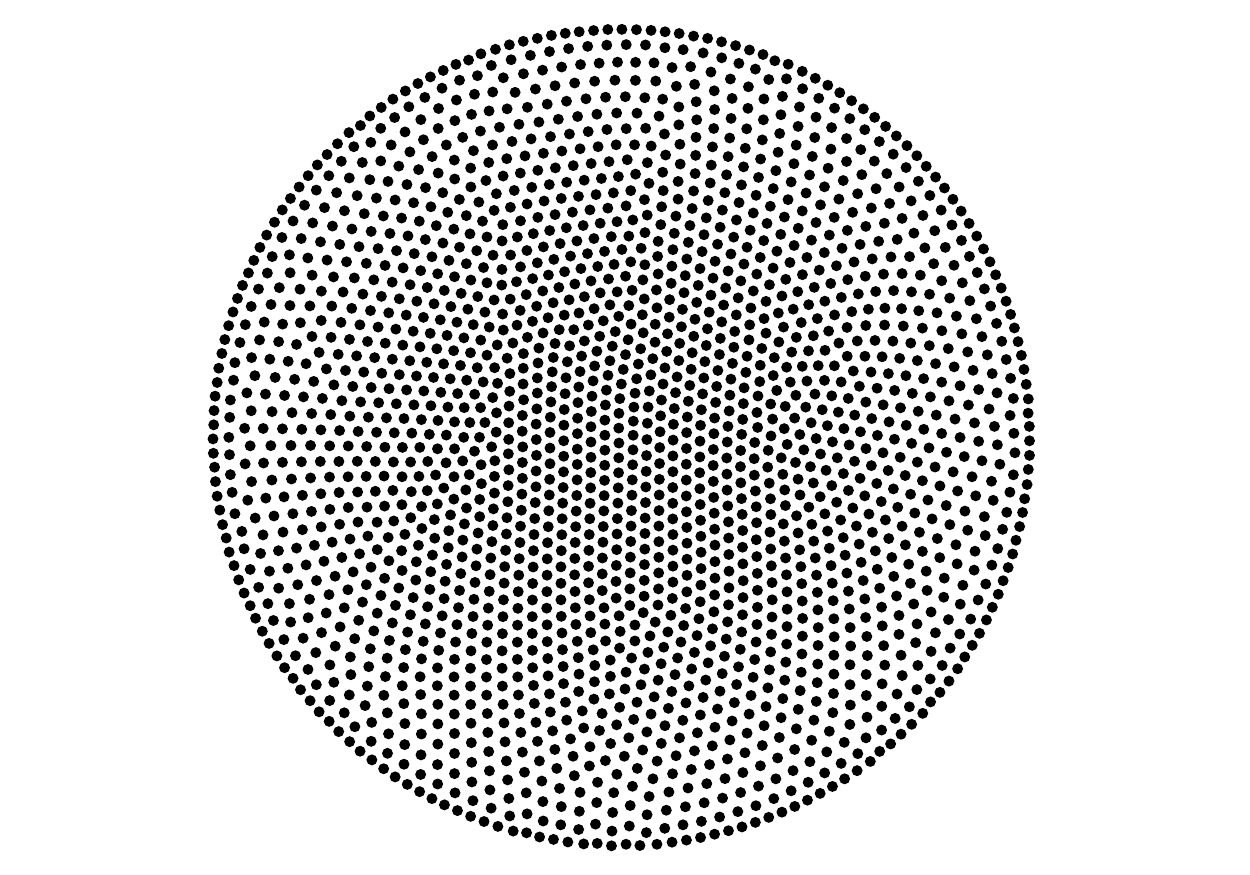}} \\\hline
$r\sim 0.5269 $ &  $r\sim 0.5480$&  $r\sim 0.5518$& $r\sim
0.5540$\\\hline
\end{tabular}
\caption{Stationary states and radius of their support as a
function of the number of particles for the potential $U(r)$ given
in equation \eqref{eq:potloglog}.} \label{tab:catmorse}
\end{table}

The results for this case are shown in Table~\ref{tab:catmorse}.
In contrast to the H-stable case, the radius of the support
converges to a limiting value. In
Figure~\ref{fig:evoradCATmorse}(a) we observe how the radius of
the support decreases in time to a limiting value with $n=1000$
particles, and in Figure~\ref{fig:evoradCATmorse}(b) we show how
the radius of the support of the stationary state increases and
converges to a certain value as a function of the number of
particles. We conclude that there should be confinement properties
for the Morse potential in the catastrophic case.

\begin{figure}[h]
\centering
\subfigure[]{\includegraphics[scale=0.36]{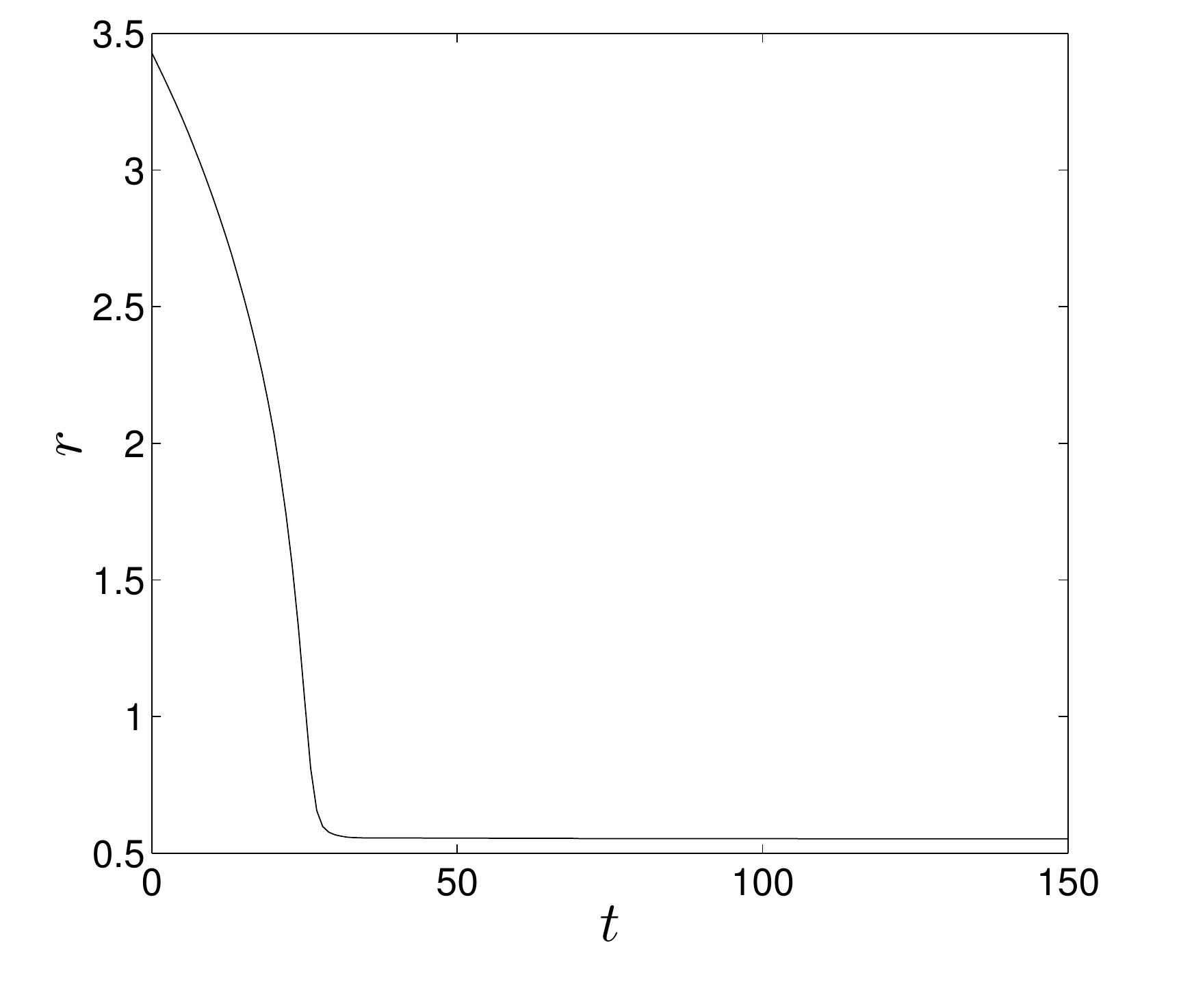}}
\subfigure[]{\includegraphics[scale=0.33]{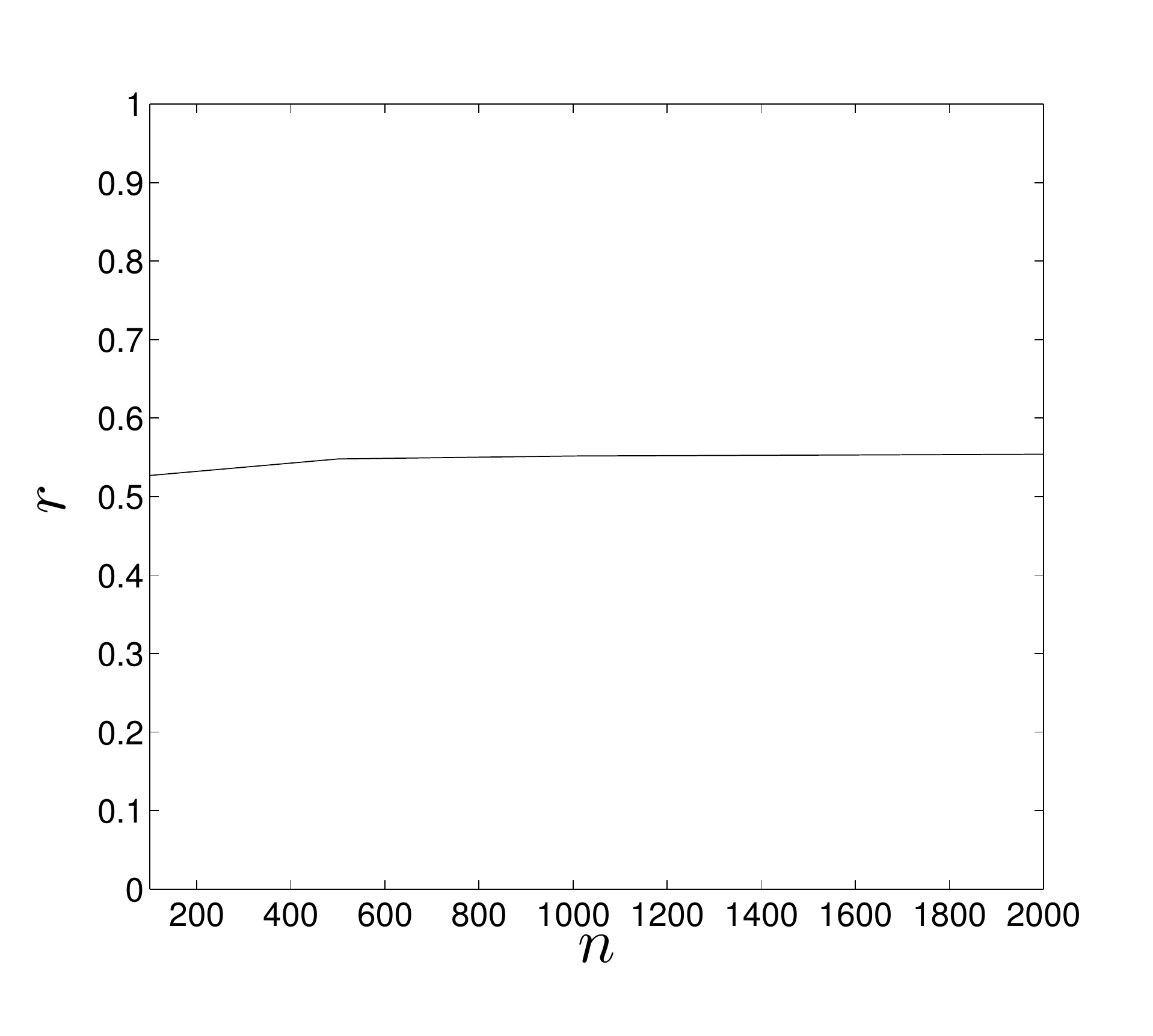}}
\caption{Catastrophic Case: (a) Evolution of the radius as a
function of $t$ for $n=1000$. (b) Evolution of the radius as a
function of $n$.} \label{fig:evoradCATmorse}
\end{figure}

The final goal would be to find replacements for the condition
\textbf{(NL-CONF)} in order to include the cases where $w(r)\to 0$
as $r\to \infty$. One possibility is to invoke scaling limits for
integrable potentials. More precisely, we scale the potential in
\eqref{pde} as
\begin{equation}
\rho_t = \nabla \cdot (\rho (\nabla W_\epsilon*\rho)), \qquad x
\in \R^N, t>0, \label{pdescaled}
\end{equation}
in such a way that $W_\epsilon (x)=\epsilon^{-N}W(x/\epsilon)$
approximates a Dirac Delta at 0 with certain weight as $\epsilon
\to 0$. Now, if the potential is such that
$$
\alpha:=\int_{\R^N} W(x)\,dx \,,
$$
then equation \eqref{pdescaled} is formally approaching
\begin{equation*}
\rho_t = \alpha \nabla \cdot (\rho \nabla \rho), \qquad x \in
\R^N, t>0.
\end{equation*}
In the H-stable case, the Morse potential satisfying $\alpha>0$
leads to a limiting nonlinear diffusive equation, which is
coherent with the no confinement property. In the catastrophic
case, the Morse potential satisfying $\alpha <0$ leads to a
limiting anti-diffusive nonlinear equation, which might also be
coherent with the confinement property of the potential. We
conjecture these integrability conditions might have some
implications for confinement properties of potentials.

\subsection*{Acknowledgments}
DB and JAC were supported by the projects  Ministerio de Ciencia e
Innovaci\'on MTM2011-27739-C04-02 and 2009-SGR-345 from Ag\`encia
de Gesti\'o d'Ajuts Universitaris i de Recerca-Generalitat de
Catalunya. JAC acknowledges support from the Royal Society through
a Wolfson Research Merit Award. YY was partially supported by NSF grant DMS-0970072. The authors would like to thank
Thomas Laurent for fruitful discussions.

\bibliographystyle{plain}
\bibliography{refs2}

\end{document}